\documentclass{IEEEtran}
\usepackage{cite}
\usepackage{amsmath,amssymb,amsfonts}
\usepackage{graphicx}
\usepackage{algorithm}
\usepackage{hyperref}
\hypersetup{hidelinks=true}
\usepackage{textcomp}
\usepackage{dsfont}

\usepackage{algpseudocode}

\newtheorem{thm}{\bf Theorem }
\newtheorem{prop}[thm]{\bf Proposition}
\newtheorem{cor}[thm]{\bf Corollary}
\newtheorem{lem}[thm]{\bf Lemma }
\newtheorem{deff}[thm]{\bf Definition}
\newtheorem{rem}[thm]{\bf Remark }
\newtheorem{exm}[thm]{\bf Example}
\newtheorem{alg}{\bf Algorithm}

\newenvironment{proof}{\par\noindent{\bf Proof\ }}{\hfill\BlackBox\\[2mm]}

\newcommand{\A}{\mathcal{A}}
\renewcommand{\L}{\mathcal{L}}
\newcommand{\X}{\mathcal{X}}

\newcommand{\D}{\mathcal{D}}
\renewcommand{\S}{\mathcal{S}}
\newcommand{\I}{\mathcal{I}}
\renewcommand{\H}{\mathcal{H}}
\newcommand{\B}{\mathcal{B}}
\newcommand{\rr}{\mathcal{R}}
\renewcommand{\S}{\mathcal{S}}

\newcommand{\K}{\mathcal{K}}
\newcommand{\T}{\mathcal{T}}

\newcommand{\F}{\mathcal{F}}
\newcommand{\R}{\mathbb{R}}
\newcommand{\N}{\mathbb{N}}
\newcommand{\Z}{\mathbb{Z}}
\newcommand{\C}{\mathbb{C}}
\newcommand{\Fb}{\mathbf{F}}

\newcommand{\Tb}{\mathbf{T}}
\newcommand{\tT}{\widetilde{\mathcal{T}}}
\newcommand{\xb}{\mathbf{x}}
\newcommand{\vb}{\mathbf{e}}
\newcommand{\xeq}{x_{\operatorname{eq}}}
\newcommand{\roa}{\mathcal{D}}
\newcommand{\pk}{\mathfrak{p}}
\newcommand{\xk}{\mathfrak{x}}
\newcommand{\Zk}{\mathcal{Z}}
\newcommand{\zk}{\mathfrak{z}}
\newcommand{\yk}{\mathfrak{y}}
\newcommand{\hk}{\mathfrak{h}}
\newcommand{\uk}{\mathfrak{u}}
\newcommand{\vk}{\mathfrak{v}}
\newcommand{\Lk}{\mathfrak{L}}
\newcommand{\tl}{\tilde{\lambda}}

\newcommand{\tzeta}{\widetilde{\zeta}}
\newcommand{\bzeta}{\bar{\zeta}}

\newcommand{\hphi}{\hat{\phi}}
\newcommand{\hV}{\hat{V}}
\newcommand{\hU}{\hat{U}}
\newcommand{\hW}{\hat{W}}
\newcommand{\hv}{\hat{v}}
\newcommand{\hu}{\hat{u}}
\newcommand{\eps}{\varepsilon}
\newcommand{\tta}{\vartheta}

\newcommand{\Nb}{\mathbf{N}}

\newcommand{\unn}{U_{\operatorname{NN}}}
\newcommand{\uedmd}{U_{\operatorname{ZK}}}
\newcommand{\vect}{\operatorname{vec}}
\newcommand{\chie}{\chi_{\operatorname{e}}}
\newcommand{\chid}{\chi_{\operatorname{d}}}
\newcommand{\ra}{\rightarrow}
\newcommand{\zero}{\mathbf{0}}
\newcommand{\id}{\operatorname{id}}
\newcommand{\dom}{\operatorname{dom}}
\newcommand{\inte}{\operatorname{int}}
\newcommand{\loclip}{\textbf{\textit{LocLip}}}
\newcommand{\lip}{\textbf{\textit{Lip}}}
\newcommand{\cj}{\wedge}

\newcommand{\set}[1]{\left\{#1\right\}}
\newcommand{\oball}{\mathcal{B}}

\newcommand{\abs}[1]{\left\vert #1 \right\vert}

\newcommand{\BlackBox}{\rule{1.5ex}{1.5ex}}    
\newcommand{\pfbox}{\hfill\BlackBox}    
\newcommand{\Qed}{\hfill$\diamond$} 
\newcommand{\ym}[1]{{\color{black} #1}}
\newcommand{\ymr}[1]{{\color{black} #1}}

\usepackage{tikz}
\usepackage{subcaption}
\usepackage{booktabs}
\usepackage{makecell}

\def\BibTeX{{\rm B\kern-.05em{\sc i\kern-.025em b}\kern-.08em
    T\kern-.1667em\lower.7ex\hbox{E}\kern-.125emX}}
\begin{document}
\title{Learning Regions of Attraction in Unknown Dynamical Systems via Zubov-Koopman Lifting: Regularities and Convergence}
\author{Yiming Meng, Ruikun Zhou, and Jun Liu, \IEEEmembership{Senior Member, IEEE}
\thanks{This research was supported in part by an NSERC Discover Grant, an
Ontario Early Researcher Award, and the Canada Research Chairs program.
This research was enabled in part by support provided by the Digital
Research Alliance of Canada (alliance.ca).}
\thanks{The authors are with the Department of Applied Mathematics, Faculty of
Mathematics, University of Waterloo, Waterloo, Ontario N2L 3G1, Canada. Emails: yiming.meng@uwaterloo.ca,
ruikun.zhou@uwaterloo.ca, j.liu@uwaterloo.ca.}
}

\maketitle

\begin{abstract}
The estimation for the region of attraction (ROA) of an asymptotically stable equilibrium point is crucial in the analysis of nonlinear systems. There has been a recent surge of interest in estimating the solution to Zubov's equation, whose non-trivial sub-level sets form the exact ROA. In this paper, we propose a lifting approach to map observable data into an infinite-dimensional function space, which generates a flow governed by the proposed `\emph{Zubov-Koopman}' operators. By learning a Zubov-Koopman operator over a fixed time interval, we can indirectly approximate the solution to Zubov's equation through iterative application of the learned operator on certain functions. We also demonstrate that a transformation of such an approximator can be readily utilized as a near-maximal Lyapunov function. We approach our goal through a comprehensive investigation of the regularities of Zubov-Koopman operators and their associated quantities. Based on these findings, we present an algorithm for learning Zubov-Koopman operators that exhibit strong convergence to the true operator. We show that this approach reduces the amount of required data and can yield desirable estimation results, as demonstrated through numerical examples. 
\end{abstract}


\section{Introduction}
\label{sec:introduction}
\IEEEPARstart{A}{n} important aspect in dynamical systems is the determination of the region of attraction (ROA) for an equilibrium point. This is of significant importance in safety-critical industries such as aviation, robotics, and power systems, where comprehending operational boundaries is crucial.  

Estimating the ROA for nonlinear systems can be rigorously resolved using formal methods \cite{li2020robustly}. \ymr{This is achieved through space discretization, formulating an `inclusion' of system transitions that serves as a symbolic abstraction, and using this surrogate abstraction to conservatively identify those states that will always reach an equilibrium point (which is the definition of the ROA).} 
\ymr{Formal methods require rigorous analysis to ensure the precision of the estimation,} and face severe computational complexity arising from state space discretization. They are also not prepared for predicting the ROA of systems with limited knowledge.

This issue can be resolved by Lyapunov methods. Lyapunov functions qualitatively characterize stability properties for various nonlinear systems, and a forward-invariant sublevel set within their domains serves as an estimate of the ROA. The existence of Lyapunov functions is guaranteed by converse Lyapunov theorems \cite{wilson1969smoothing,lin1996smooth, clarke1998asymptotic,teel2000smooth}. 
The primary challenge lies in the construction of Lyapunov functions that possibly enhance the conservative estimation of ROAs for more general nonlinear systems.  On the other hand, Zubov's theorem characterizes the maximal
Lyapunov function \cite{vannelli1985maximal} defined on the domain of attraction. However, it requires solving a partial
differential equation (PDE) \cite{hassan2002nonlinear}. 

In real-world applications, limited knowledge of the system dynamics can make the identification of ROAs using Lyapunov methods even more challenging. Inspired by recent advances in Koopman operator-based system identification for unknown dynamical systems,   we propose a Zubov-Koopman lifting approach to estimate the ROA by approximating the solutions to Zubov's equations, whose non-trivial sublevel sets form the
exact ROA.  We aim to: 1) study the regularities of Zubov's equation and introduce a linear Zubov-Koopman operator to characterize the solution; 2) configure the trajectory data and learn the Zubov-Koopman operator; 3) approximate the entire  ROA for an asymptotically stable \ymr{equilibrium point}, with emphasis on an asymptotically stable equilibrium point, using the learned Zubov-Koopman operator; 4)  approximate near-maximal Lyapunov functions and provide formal verification as a byproduct.


We review some crucial results from the literature that are pertinent to the work presented in this paper.

\subsection{Related Work}
The computation of Lyapunov functions has a long history \cite{giesl2015review} and has gained increased attention with respect to data-driven methods~\cite{dawson2022safe}. In particular, the development of the Koopman operator theory provides a promising alternative learning approach for nonlinear system identification and Lyapunov function constructions~\cite{brunton2021modern, koopman1931hamiltonian, mezic2005spectral}. 

In essence, Koopman operators simplify the nonlinear analysis by lifting the states of the system into the space of observable functions, which evolve linearly governed by Koopman operators. The spectral properties of linear Koopman operators can facilitate Koopman mode decomposition for nonlinear systems using a specific set of observable functions. Several established techniques, such as dynamic mode decomposition (DMD)~\cite{schmid2009dynamic} and extended dynamic mode decomposition (EDMD) ~\cite{williams2015data}, can be employed to acquire this linear representation. In addition, the autoencoder architecture can be considered as suitable neural-network observable functions, although it requires more training efforts to achieve an optimized linear representation~\cite{lusch2018deep, azencot2020forecasting}. 

Taking advantage of the spectral representation, in~\cite{mauroy2013spectral} and \cite{mauroy2016global}, the authors 
showed that a set of Lyapunov functions for nonlinear systems with global stability can be constructed based on the eigenfunctions of the learned Koopman operator. The work in ~\cite{deka2022koopman} improved the multi-step trajectory prediction accuracy by a modified autoencoder architecture, with the corresponding Koopman eigenfunctions parameterizing a set of Lyapunov function candidates. In this framework, compared to the non-Koopman neural Lyapunov framework \cite{zhou2022neural}, it becomes possible to achieve a more desirable ROA estimation by exploring various valid combinations within the set of resulting Lyapunov functions. 

\ymr{The Koopman method in \cite{deka2022koopman, mauroy2013spectral, mauroy2016global, yi2023equivalence, bierwart4569180numerical, zeng2024sampling} requires working on a forward-invariant compact subset within the state space for stability analysis. } 
This allows the Koopman operators to preserve the associated function space on the invariant set. However, it sacrifices the ability to construct Lyapunov functions on a larger scale, given their unbounded nature near the boundary of the open ROA. 

Considering this limitation, Zubov's construction of a Lyapunov function seems to address the issue by ensuring that the solution is always bounded. This approach can offer potential advantages in numerical approximations, particularly when solving a PDE to find a Lyapunov function. This property enables the extension of the function domain to the entire state space or any desired set where computations occur \cite{liu2023towards}. 

In this regard, recent investigations have focused on numerical solutions to Zubov's equation. In cases where system dynamics are known, the approach in \cite{grune2021computing} utilized local exponential stability conditions to train neural networks, similar in nature to physics-informed neural networks (PINNs) \cite{raissi2019physics}. The most recent work \cite{liu2023towards} introduced PINN algorithms for computing Lyapunov functions capable of approximating the entire ROA for an asymptotically stable compact set with high accuracy. The Lyapunov function candidate generated by the neural network was also formally verified.  The work in \cite{kang2021data} 
employed a purely data-driven approach for estimating the solution to Zubov's PDE. This method, however,  requires long-term observation of trajectory data, and may have limited predictability of ROAs when the observation time is restricted.

\subsection{Contributions}
Considering the pros and cons of the Koopman approach and Zubov's PDE, we propose a Zubov-Koopman operator-based data-driven technique, leveraging trajectory data within a fixed observation span to predict the ROA of a known asymptotically stable compact set or equilibrium point for unknown dynamics.  In brief:
\begin{enumerate}
    \item We conduct a regularity analysis for Zubov's equation, where we relax the assumption from continuously differentiable vector fields to local Lipschitz continuity. Additionally, we explore how the solution to Zubov's equation connects with a more general notion of Lyapunov functions. 
    \item We introduce a Zubov-Koopman operator and establish its connections with Zubov's equation. Building upon the regularity analysis mentioned earlier, we provide a detailed proof demonstrating how a Zubov-Koopman operator can lead to convergence to the solution of Zubov's equation in a more general sense. 
    \item While it is often assumed that finite-rank approximations apply to Koopman-like operators, we rigorously investigate the theoretical feasibility of finite-dimensional approximations for Zubov-Koopman operators and demonstrate how they can guide the selection of observable functions. Notably, \ymr{the existing theorems on the existence of Koopman eigenfunctions are based either on classical linearization theorems \cite{lan2013linearization}\cite[Section IV]{ mauroy2013spectral}, which may not be suitable for locally Lipschitz continuous vector fields, or are complicated to verify \cite[Proposition 6]{kvalheim2021existence}.} 
    The spectral analysis in this paper will focus on the specific goal of approximating the solution to Zubov's equation under mild conditions.
    \item We introduce a learning algorithm for the Zubov-Koopman operator and employ the learned operator to approximate the solution to Zubov's equation, thereby estimating the ROA. We show that extra efforts are required to obtain a near-maximal Lyapunov function. 
    \item We provide numerical experiments and demonstrate the effectiveness of the proposed approach. 
\end{enumerate}

The rest of the paper is organized as follows. Section \ref{sec: pre} presents some preliminaries on the Koopman operator, Zubov's equation, and concepts for analyzing its solution regularity.   Section \ref{sec: regularity} focuses on conducting regularity analysis under mild system conditions.  Building on this, in Section \ref{sec: zubov-Koopman}, we introduce the Zubov-Koopman operator and demonstrate how the Zubov-Koopman operator can effectively characterize the solution to Zubov's equation. Section \ref{sec: finite-approx} establishes the theoretical feasibility of a finite-dimensional representation of the Zubov-Koopman operator and demonstrates its interplay with observable data. Section \ref{sec: algrithms} introduces the algorithms, with numerical experiments conducted in Section \ref{sec: num}.  The paper is concluded in Section \ref{sec: conclu}. 

\subsection{Notation}
We denote by $\R^n$ the Euclidean space of dimension $n>1$, and by $\R$ the set of real numbers. 
For $x\in\R^n$ and $r\ge 0$, we denote the ball of radius $r$ centered at $x$ by $\B(x, r)=\set{y\in\R^n:\,\abs{y-x}\le r}$, where $\abs{\;\cdot\;}$ is the Euclidean norm. For a closed set $A\subset\R^n$ and $x\in\R^n$, we denote the distance from $x$ to $A$ by $\abs{x}_{A}=\inf_{y\in A}\abs{x-y}$ and $r$-neighborhood of $A$ by $\B(A, r)=\cup_{x\in A}\B(x, r)=\set{x\in\R^n:\,\abs{x}_A\le r}$. For a set $A\subseteq\R^n$, $\overline{A}$ denotes its closure, 
$\inte(A)$ denotes its interior, and $\partial A$ denotes its boundary.  For two sets $A,B\subseteq\R^n$, the set difference is defined by $A\setminus B=\set{x:\,x\in A,\,x\not\in B}$. 
For finite-dimensional matrices, we use the Frobenius norm  $\|\cdot\|_F$ as the metric.  Given $a,b\in\R$, we define $a\wedge b:=\min(a,b)$. 

Let $C(\Omega)$ be the set of continuous functions and $C_b(\Omega)$   be the set of bounded continuous functions with domain $\Omega$. We denote the set of $i$-th continuously
differentiable functions by $C^i(\Omega)$, and similarly, bounded continuously
differentiable functions by $C_b^i(\Omega)$. We denote the set of locally and uniformly Lipschitz continuous functions by $\loclip(\Omega)$ and $\lip(\Omega)$. When making general statements for $v\in C^1(\R^n)$ with $n\geq 1$, we denote $\nabla $ as its gradient (or as the derivative when $n=1$).

\section{Preliminaries}\label{sec: pre}
\subsection{Dynamical Systems}
Given a state space $\X\subseteq\R^n$, 
we consider a continuous-time nonlinear dynamical system of the form 

\begin{equation}\label{E: sys}
    \dot{\xb}(t) = f(\xb(t)),\;\;\xb(0)=x\in\X,\;t\in[0,\infty), 
\end{equation}
where $x$ denotes the initial condition, and the vector field $f:\X\ra\X$ is assumed to be locally Lipschitz. 

On the maximal  interval of existence $\mathcal{I}\subseteq [0,\infty)$, the forward flow map (solution map) $\phi: \I\times \X\rightarrow \X$ should satisfy 
\begin{equation}
    \begin{cases}
        &\partial_t(\phi(t,x)) = f(\phi(t,x)), \\
        & \phi(0, x)=x, \\
        & \phi(s, \phi(t,x))=\phi(t+s, x), \;\forall t,s\in\mathcal{I}
    \end{cases}
\end{equation}

Without loss of generality, 
throughout the paper, we will assume that  the 
maximal interval of existence of the (unique) flow map to the initial value problem   \eqref{E: sys} is $\mathcal{I}=[0,\infty)$. We also generally consider that the state space $\X=\R^n$, i.e., the flow is not necessarily assumed to be invariant within a strict subset of $\R^n$.

\ymr{To define Koopman operators, which govern the evolution of observable functions, let us now consider a complete 
function space $\mathcal{F}$ of the observable real-valued functions $h: \X\rightarrow\mathbb{R}$.} 

\begin{deff}[Koopman Operator]
    \label{def: Koopman} 
The Koopman operator family $\{\K_t\}_{t\geq 0}$ of system \eqref{E: sys} is a collection of maps $\mathcal{K}_t: \mathcal{F} \rightarrow \mathcal{F}$  defined by
\begin{align}
\mathcal{K}_t h = h\circ \phi(t, \cdot), \quad h \in \mathcal{F}
\end{align}
for each $t\geq 0$, where $\circ$ is the composition operator. The (infinitesimal) generator $\L_f$ of $\{\K_t\}_{t\geq 0}$ is defined by
\begin{equation}
    \L_f h(x):= \lim_{t\ra 0}\frac{\K_t h(x)-h(x)}{t},
\end{equation}
where the observable functions should be within the domain of $\L_f$, i.e. $\dom(\L_f)=\set{h\in\F:\lim_{t\ra 0}\frac{\K_t h(x)-h(x)}{t}\;\text{exists}}. $
\end{deff}

Suppose that the observable functions
are bounded and continuously differentiable, the  generator is such that
\begin{equation}
    \L_f h = \nabla h \cdot f, \;\;h\in C_b^1(\X).
\end{equation}

Koopman operators form a linear $C_0$-semigroup that satisfies the following criteria. They allow us to study the nonlinear dynamics through the infinite-dimensional lifted space of observable functions with linear dynamics.

\begin{deff}[Semigroup]\label{def: semigroup}
A one-parameter family $\{\S_t\}_{t\geq 0}$,
of bounded linear operators from $\F$ into $\F$ is a semigroup of
bounded linear operators on $\F$ if
\begin{enumerate}
    \item $\S_0=\id$, ($\id$ is the identity operator). 
    \item $\S_t\circ \S_s=\S_{t+s}$ for every $t,s\geq 0$. 
\end{enumerate}
In addition, a semigroup $\{\S_t\}_{t\geq 0}$ is a strongly continuous semigroup, or $C_0$-semigroup, if 
$\lim_{t\downarrow 0} \S_t h = h$ for all $h \in \F$. 
\end{deff}

\subsection{Concept of Stability}

We are interested in systems of the form \eqref{E: sys} with an intrinsic asymptotically stable set $\A\subseteq\X$. We define the set stability as follows. 
\begin{deff}[Set stability]
A closed invariant set $\A\subseteq\mathcal{X}$ is
said to be  asymptotically stable for \eqref{E: sys}
if 
\begin{enumerate}
    \item  for every $\eps>0$, there exists  a $\delta>0$ such that $|x|_\A< \delta$ implies $|\phi(t, x)|_\A <\eps$ for all $t\geq 0$, and
    \item there exists a $\delta>0$ such that $|x|_\A< \delta$ implies $\lim_{t\ra\infty}|\phi(t, x)|_\A=0$. 
\end{enumerate}

Furthermore, $\A$ is said to be locally exponentially stable, if there exists a $\delta>0, M>0$ and $c>0$ such that $|\phi(t, x)|_\A\leq M|x|_\A e^{-ct}$, for all $t\geq 0$ and $x\in\set{x\in\X: |x|_\A\leq \delta}. $
\end{deff}

We further define the region of attraction (ROA) of $\A$ given its asymptotic stability, 
 which quantifies a region of the state space from which each absolutely continuous trajectory starts and eventually converges to the
attractor itself.

\begin{deff}[ROA]
Suppose that $\A$ is asymptotically stable, the ROA of $\A$ is a set  defined as
$\roa(\A):=\set{x\in\X: \lim_{t\ra\infty}|\phi(t, x)|_\A=0}.$
\end{deff}

\begin{rem}
    It is a well-known result that the ROA is an open and forward invariant set. \Qed
\end{rem}

To better convey the idea of this paper, we propose the following hypothesis for unknown systems in the form of \eqref{E: sys}.
\begin{itemize}
    \item[(H1)] We assume that there exists an equilibrium point $\xeq$ of \eqref{E: sys}, i.e. a point such that $f(\xeq)=\zero$. 
    \item[(H2)]We assume full knowledge of $\xeq$, and that $\set{\xeq}$ is locally exponentially stable. 
\end{itemize}
Based on Hypotheses (H1) and (H2), the purpose of this paper is to employ a data-driven approach to estimate the ROA of an equilibrium point of unknown systems.

\subsection{Zubov's Theorem}
The ROA can be characterized by a maximal Lyapunov function as described in the following theorem \cite{vannelli1985maximal}. 
\begin{thm}\label{thm: max_lyap}
	Let $D\subseteq\R^n$ be an open set. Suppose that there exists a function $V\in C^1(D)$  such that $\A\subseteq D$ and the following conditions hold: 1) $V$ is positive definite on $D$ with respect to $\A$, i.e., $V(x)=0$ for all $x\in \A$ and $V(x)>0$ for all $x\in D\setminus \A$; 2) the derivative of $V$ along solutions of (\ref{E: sys}) 
     is well-defined for all $x\in D$ and satisfies 
          \begin{equation}\label{eq:DV1}
     \nabla V(x)\cdot f(x)  = - q(x),    
     \end{equation}
where the function \ym{$q:\,\R^n\ra\R$} is continuous and positive definite with respect to $\A$; 3) $V(x)\ra \infty$ as $x\ra \partial D$ or $|x|\ra \infty$. 
 Then $D=\roa(\A)$. 
\end{thm}

Theorem \ref{thm: max_lyap} is equivalent to Zubov's theorem \cite{zubov1961methods}  stated below. 

\begin{thm}\label{thm: zubov}
Let $D\subset\R^n$ be an open set containing $\A$. Then $D=\roa(\A)$ if and only if there exists two continuous functions $W:\,D\ra \R$ and \ym{$\eta:\,\R^n\ra \R$} such that the following conditions hold:
 \begin{enumerate}
     \item $0<W(x)<1$ for all $x\in D\setminus\A$ and $W(x)=0$ for all $x\in \A$; 
     \item$ \eta$ is positive definite on $D$ with respect to $\A$;  
     \item for any sufficiently small $c_3>0$, there exist $c_1, c_2>0$ such that $|x|_\A\ge c_3$ implies 
     $W(x)>c_1$ and $\eta(x)>c_2$;  
     \item $W(x)\ra 1$ as $x\ra y$ for any $y\in \partial D$;
     \item $W$ and $\eta$ satisfy (in the conventional differentiable sense)
        \begin{equation}\label{E: zubov}
     -\nabla W(x)\cdot f(x) +\eta(x)(1-W(x)) = 0.
     \end{equation} 
 \end{enumerate}
\end{thm}
\begin{rem}
    Based on Hypothesis (H1) and (H2), a typical choice for $\eta$ is given by $\eta(x)=\alpha|x-\xeq|$ for some $\alpha>0$. 
In addition, 
Theorem \ref{thm: max_lyap} and Theorem \ref{thm: zubov} can be related by the following equation
    \begin{equation}\label{E: V2W}
    W(x) = 1 - \exp(-\alpha V(x)), \;\;x\in\roa(\A),
\end{equation}
for some constant $\alpha>0$.  It is easy to verify that $V$ satisfying (\ref{eq:DV1}) implies $\nabla W(x)\cdot f(x)  = -\alpha(1-W(x)) q(x)$, 
 which verifies (\ref{E: zubov}) with $\eta(x)=\alpha q(x)$. \Qed
 \end{rem}

Note that $W$ remains bounded, with its value approaching $1$ as $x$ approaches the boundary. In contrast, the function $V$ in Theorem \ref{thm: max_lyap} approaches infinity as $x$ approaches the boundary. The boundedness can offer significant advantages in data-driven approximations. In particular, it allows us to focus on a desired region of interest where computations occur. However, this modification is not as straightforward as it might appear. To better understand how Zubov's construction can guide the data-driven approach to learning the ROA, we introduce an extension of Zubov's theorem in Section \ref{sec: regularity} and rigorously examine the regularity of the solutions.

\subsection{Differentiability and Viscosity Solution}
As we will see in Section \ref{sec: regularity}, the Zubov equation may not always possess differentiable solutions. One may need to consider a more general sense of solution to resolve. 

\begin{deff}[Viscosity solution]\label{def: vis1}
Define the superdifferential and the subdifferential sets of $v\in C(\R^n)$ at $x$ respectively as
\begin{small}
    \begin{subequations}
\begin{align}
& \partial^+v(x) 
=\left\{p\in \R^n: \;\limsup_{y\ra x}\frac{v(y)-v(x)-p\cdot (y-x)}{|y-x|}\leq 0\right\}\label{E: compare1a},\\
& \partial^-v(x) 
= \left\{q\in \R^n: \;\liminf_{y\ra x}\frac{v(y)-v(x)-q\cdot (y-x)}{|y-x|}\geq 0\right\}\label{E: compare1b}.
\end{align}
\end{subequations}
\end{small}

\noindent A continuous function $v$ of a PDE of the form $F(x,v(x),\nabla v(x))=0$ (possibly encoded with boundary conditions) is a viscosity solution if the following conditions are satisfied:
\begin{itemize}
    \item[(1)] (viscosity subsolution) $F(x, v(x), p)\leq 0$ for all $x\in\R^n$ and for all $p\in \partial^+v(x)$.
    \item[(2)] (viscosity supersolution) $F(x, v(x), q)\geq 0$ for all $x\in\R^n$ and for all $q\in \partial^-v(x)$.
\end{itemize}
\end{deff}
\begin{rem}
More details of viscosity solutions are provided in Appendix \ref{sec: app1}. The concept of viscosity solution relaxes $C^1$ solutions. Note that, at differentiable points,  $\nabla v(x)$ exists and $\{\nabla v(x)\}=\partial^+v(x) = \partial^-v(x)$. In this case, to justify a viscosity solution, we can simply substitute $\nabla v(x)$ and check if $F(x, v(x), \nabla v(x))=0$. 
If $v$ is not differentiable at a given point, then we have to go through the comparisons (1) and (2) in Definition \ref{def: vis1} to verify that $v$ is a viscosity solution.  \Qed 
\end{rem}

\section{An Extension of Zubov's Theorem and Regularity Analysis}\label{sec: regularity}
For the rest of this paper, we will assume that Hypothesis (H1) and (H2) hold and denote $\mathcal{A}:=\{\xeq\}$.

Inspired by \eqref{E: zubov} and its connection with $V$ through Eq. \eqref{E: V2W},  to facilitate data-driven approximation, we explore the following dual form of Zubov's equation \eqref{E: zubov}, namely Zubov's dual equation, 
    \begin{equation}\label{E: dual}
H(x, U(x), \nabla U(x))=0, \;U(\xeq)= 1, 
\end{equation}
where 
\begin{equation}\label{E: H}
    H(x, U(x), p):=- p\cdot f(x)+\eta(x)U(x),\;\;p\in\R^n. 
\end{equation}
For now, we intentionally omit specifying the domain of $U$ and would like to discuss regularity analysis later in this section. 

Note that, on $\roa(\A)$, we always have $U(x)=1-W(x)$. \ymr{We prefer this dual form of Zubov's equation because it allows us to potentially use a time series to approximate the solution through iterations (thereby reducing the amount of required data), as demonstrated in  Section \ref{sec: zubov-Koopman}}.

\subsection{Solution to Zubov's Dual Equation}\label{sec: Zubov_dual}
In this section, we construct the solution to the dual equation \eqref{E: dual} of Zubov's equation and demonstrate that, in a general context, it satisfies \eqref{E: dual} in a viscosity sense. 

Let $\eta$ be positive definite w.r.t. $\A$. Define 
   \begin{equation}
    \label{E: V}
V(x)=\int_0^\infty \eta(\phi(t, x))dt,\quad x\in \R^n, 
\end{equation} 
where if the integral diverges, we let $V(x)=\infty$. Then, we have the following property \cite[Proposition 1]{liu2023towards}. 

\begin{lem}\label{lem: V}
The function $V:\,\R^n\ra\R\cup\set{\infty}$ defined by \eqref{E: V} satisfies the following: 
\begin{enumerate}
    \item $V(x)<\infty$ if and only if $x\in \roa(\A)$;
    \item $V(x)\ra \infty$ as $x\ra \partial \roa(\A)$; 
    \item $V$ is positive definite with respect to $\A$. 
\end{enumerate}
\end{lem}

\begin{deff}
Let $h\in C_b(\R^n)$. We further define 
    \begin{equation}\label{E: U}
    U_h(x)=
    \begin{cases}
        &\exp\set{-V(x)}h(\phi_\infty(x)),\;\text{if}\;V(x)<\infty,\\
        & 0,\;\;\text{otherwise},
    \end{cases}
\end{equation}
 where $\phi_\infty(x):=\lim\limits_{t\ra\infty}\phi(t, x)$. For test function $h(x)=\mathds{1}(x)$, where $\mathds{1}(x)\equiv 1$ for all $x$,  we denote $U:=U_\mathds{1}$ for simplicity. 
\end{deff}

\begin{rem}
    The notion of $h(\phi_\infty(x))$ seems redundant in that, given $V(x)<\infty$, by 1) of Lemma \ref{lem: V}, we have $x\in\roa(\A)$ and $h(\phi_\infty(x))=h\circ\lim\limits_{t\ra\infty}\phi(t, x)=h(\xeq)$. However, we still keep it in \eqref{E: U} for consistency when using a time series to approximate in Section \ref{sec: zubov-Koopman}. \Qed
\end{rem}

\ymr{We would like to show that the construction in \eqref{E: U} is a solution to \eqref{E: dual}. Getting the conventional continuously differentiable solutions to \eqref{E: dual} or \eqref{E: zubov} depends on the differentiability of $f$ (and hence the differentiability of $\phi(t, x)$ w.r.t. $x$) and $\eta$. However, this may not always be the case given the general assumptions on $f$ and $\eta$. In view of numerical approximations, it is not appropriate to construct a convergent sequence with respect to the $C^1$ uniform norm (e.g. \cite[Chapter 3]{farsi2023model}). Based on this viscosity property analysis, this paper provides guidance for constructing numerical solutions through a Zubov-Koopman operator learning framework in subsequent sections.} We provide a simple example below to demonstrate that we necessarily need to consider viscosity solutions to \eqref{E: dual}. 

\begin{exm}\label{eg: 1d}
Consider a simple dynamical system
    $$\dot{\xb}(t)=-\xb(t) + \xb^3(t), \;\;\xb(0)=x, $$
and a convex, positive definite function $\eta(x)=|x|$.\ymr{Then, it can be verified that \begin{small}
    $V(x)=\frac{1}{2}\log\left(\frac{1+|x|}{1-|x|}\right)$ for $x\in(-1,1)$,
\end{small} \begin{small}$U(x)=\sqrt{\frac{1-|x|}{1+|x|}}$\end{small} for $x\in(-1,1)$,  and $U(x)=0$ elsewhere.} Clearly, $(x-x^3)U'(x)+|x|U(x)=0$ for all $x\in\R\setminus\{0, \pm1\}$ and $U$ is not differentiable at $0$ and $\pm1$. Given the fact that $U$ is concave near $0$,  we also have that $\partial^-U(0)=\emptyset$ by \cite[Chapter II, Proposition 4.7]{bardi1997optimal}. On the other hand, $\partial^+U(0)=[-1,1]$. For all $p\in \partial^+U(0)$, we have that $H(0, U(0), p)\leq 0$, which show that $U$ satisfies \eqref{E: dual} at $0$ in a viscosity sense. Similar proofs can be done at $x=\pm 1$. \Qed
\end{exm}

\begin{lem}\label{prop: max_lyapunov}
Let $F(x,  p):=-p\cdot f(x)-\eta(x).$
    Then,  on $\roa(\A)$, the function $V(x)$ in \eqref{E: V} is a viscosity solution to $F(x, \nabla V(x))=0$ with $V(\xeq)=0$. 
\end{lem}
\begin{proof}
It is clear that $V\in C(\R^n)$. On $\roa(\A)$, by Lemma \ref{lem: V} and \cite[Proposition 1]{liu2023towards}, we have  
$V(x)<\infty$ as well as the  dynamic programming principle $V(x) =\int_0^t \eta(\phi(s, x)) ds + V(\phi(t, x))$ for all $x\in \roa(\A)$ and $t>0$. We now show that $V$ is a viscosity solution on $\roa(\A)$ using the equivalent conditions introduced in Appendix \ref{sec: app1}. 
Let $\nu\in C^1$ and $x$ be a local maximum point of $V-\nu$. Then $V(x)-V(z)\geq \nu(x)-\nu(z)$ for all $z\in\oball(x,r)$. 
For $t$ sufficiently small, we have $\phi(t,x)\in\oball(x,r)$. Therefore,
\begin{small}
\begin{equation}\label{E: small_t}
    \begin{split}
    \nu(x)-\nu(\phi(t, x))
    \leq & V(x)- V(\phi(t, x))\\
        = & \int_0^t \eta(\phi(s, x))ds + V(\phi(t,x)) - V(\phi(t,x)). 
    \end{split}
\end{equation}
\end{small}

\noindent Considering the infinitesimal behavior on both sides of \eqref{E: small_t}, i.e., dividing both sides by $t$ and taking the limit as $t\downarrow 0$,  we have $-\nabla \nu(x)\cdot f(x) \leq \eta(x)$. The case when $x$ is a local minimum of $V-\nu$ can be proved in the same manner. Then, $V$ is a viscosity solution to $F(x, \nabla V(x))=0$ on $\roa(\A)$. 

\end{proof}

\begin{rem}
From the above proof, one may notice that the sign of $F(x,p)$ (or similarly, $H(x, U(x), p)$ as in \eqref{E: dual_test}) matters. In other words, $V$ is not a viscosity solution to $-F(x,\nabla V(x))=0$. 
    \Qed
\end{rem}

\begin{thm}[Viscosity Solution]\label{thm: uniqueness_U}
For any $h\in C_b(\R^n)$, the $U_h$ defined in \eqref{E: U} is a viscosity
solution to 
\begin{equation}\label{E: dual_test}
    H(x, U_h(x), \nabla U_h(x))=0, \;\;U_h(\xeq)=h(\xeq)
\end{equation}
on $\R^n$, where $H$ is defined in \eqref{E: H}.   
\end{thm}

\begin{proof}
It is clear that $U_h\in C(\R^n)$ and $U_h(\xeq)=\exp\set{-0}h(\xeq)=h(\xeq)$. For $h$ with $h(\xeq)=0$, $U_h\equiv 0$ is automatically a differentiable and, hence, a viscosity solution to \eqref{E: dual_test}. We verify the case when $h(\xeq)\neq 0$. 

On $\roa(\A)$, since $V(x)<\infty$ by Lemma \ref{lem: V}, the inverse function ($-\log(U_h(x)/h(\xeq))$) of $U_h(x)=\exp(-V(x))h(\xeq)$ is well defined.  Since $-V$ is a viscosity solution to $-F(x, \nabla V)=0$, where $F$ is defined in Lemma \ref{prop: max_lyapunov}, by \cite[Chapter 2, Proposition 2.5]{bardi1997optimal}, we immediately have that
$U_h$ is a viscosity solution of 
$-F(x, -\log'(U_h(x)/h(\xeq))\nabla U_h(x)) = -\frac{1}{U_h(x)}\nabla U_h(x)\cdot f(x) +\eta(x)=0$. This implies that $U_h$ is a viscosity solution of  \eqref{E: dual_test} in $\roa(\A)$. For $x\in\R^n\setminus\overline{\roa(\A)}$, $u(x)\equiv 0$ is always a viscosity solution.  

It suffices to show the viscosity property on $\partial\roa(\A)$. 
Note that $U_h$ may not be differentiable on $\partial \roa(\A)$. However,  by the famous Rademacher’s Theorem \cite{evans2010partial},  $U_h$ is differentiable a.e. on each subdomain $\Omega\subseteq\R^n\setminus\roa(\A)$. 
Equivalently, 
    $U_h$ satisfies 
    \begin{equation}\label{E: dual_ae}
         \eta(x)U_h(x)+H(x, \nabla U_h(x))=0,\;\;\text{a.e. in}\; \R^n\setminus\roa(\A)
    \end{equation}
    in the conventional sense, where $ H(x, p)=-p\cdot f(x)$. 
    Now 
    we  introduce a smooth mollifier $\rho\in C_0^\infty(\R^n)$ with compact support $\oball(\zero, 1)$ 
    such that $\int_{\R^n}\rho(y)dy=1$, as well as a kernel $\rho_\eps(x)=\frac{1}{\eps^n}\rho\left(\frac{x}{\eps}\right)$ for $\eps>0$. Then we define the convolution for $u_h$ by $u_h^\eps(x): = (U_h * \rho_\eps)(x):=\int_{\R^n} U_h(y)\rho_\eps(x-y)dy$, 
    and similarly for $\nabla u_h^\eps(x)=(\nabla U_h * \rho_\eps)(x)$. It is clear that $u_h^\eps\ra U_h$ locally uniformly in $\R^n$. Convolve both side of \eqref{E: dual_ae} with $\rho_\eps$, then, 
    \begin{small}
            \begin{equation}
    \begin{split}
                &\eta(x)u_h^\eps(x)+\int_{\R^n}H(y, \nabla U_h(y)\rho_\eps(x-y)dy\\
                =&\int_{\R^n} U_h(y)(\eta(x)-\eta(y))\rho_\eps(x-y)dy, \;\;\forall x\in\R^n\setminus\roa(\A). 
    \end{split}
    \end{equation}
    \end{small}

    \noindent     Since      $H$ is also linear in $p$ for any fixed $x$, it follows that, for all $x\in\R^n\setminus\roa(\A)$, 
    \begin{small}
          \begin{equation}\label{E: converge_1}
        \begin{split}
                &\eta(x)u_h^\eps(x)+H(x, \nabla u_h^\eps(x))\\
                = &\int_{\R^n} U_h(y)(\eta(x)-\eta(y))\rho_\eps(x-y)dy\\
                & +\int_{\R^n} (H(x,\nabla U_h(y))-H(y,\nabla U_h(y)))\rho_\eps(x-y)dy. 
    \end{split}
    \end{equation}  
    \end{small}
    
    \noindent Note that  $u_h^\eps$ is differentiable, which is automatically a viscosity solution to the above equation. Since the R.H.S. of \eqref{E: converge_1} converges to $0$ uniformly, by \cite[Chapter II, proposition 2.2.]{bardi1997optimal}, it follows that $U_h$ is a  viscosity solution to \eqref{E: dual_test} on $\R^n\setminus\roa(\A)$, which completes the proof. 
\end{proof}

The following corollary shows the connection between \eqref{E: dual} and \eqref{E: zubov} in the viscosity sense. The proof follows the same procedure as outlined above. 
We hence do not repeat. 
\begin{cor}
    For any $h\in C_b(\R^n)$, let $W_h=1-U_h$, where $U_h$ is defined in \eqref{E: U}. Define 
        \begin{equation}\label{E: Z_function}
         Z(x, W_h(x), p):=- p\cdot f(x)+\eta(x)(1-W_h(x)),\;\;p\in\R^n.
    \end{equation}
Then, $W_h$ is a viscosity solution to 
       \begin{equation*}
        Z(x, W_h(x), \nabla W_h(x))=0, \;\;W_h(\xeq)=1-h(\xeq)
    \end{equation*}
on $\R^n$ if and only if $U_h$ is a viscosity solution to \eqref{E: dual_test}. This particularly holds for $h=\mathds{1}$. \end{cor}

\begin{thm}\label{thm: unique}
    Suppose that $\eta$ and $h$ are also locally Lipschitz continuous, then $U_h$ is the unique bounded viscosity solution to \eqref{E: dual_test}. 
\end{thm}

The proof for more general perturbed systems is provided in \cite[Theorem 3.8]{camilli2001generalization}. A succinct proof of Theorem \ref{thm: unique} is offered in Appendix \ref{sec: app_uniq}, based on assumptions specific to our setting.

\ymr{By the proof of Theorem \ref{thm: uniqueness_U},  and under the same assumptions as in the above theorem, it also implies the uniqueness of the viscosity solution $V$ for $-\nabla V(x)\cdot f(x)-\eta(x)=0$. The Lipschitz continuity of $V$ has also been proven in a recent work \cite{liu2023physics}.} 
\ym{\begin{rem}
    Note that this $V$ may be a Lyapunov function on $\roa(\A)$  given that $\partial^-V(x)\neq \emptyset$ for all $x\in\roa(\A)$. This follows from Definition \ref{def: vis1} and \cite{zhu2003lower},  since  we have $-p\cdot f(x)-\eta(x)\geq 0$ for all $p\in\partial^-V(x)$ and for all $x\in\roa(\A)$.  In other words, a differentiable $V$ is already a Lyapunov function, while a non-differentiable $V$ may not be. However, theoretically, there always exists a smooth version of $V$ that can be as a Lyapunov function \cite[Section 5]{camilli2001generalization}. The subsequent sections of the paper will provide a perspective on how to use data-driven Koopman-based methods to construct a Lyapunov function, serving as an alternative to the sum-of-squares approach  in \cite{jones2021converse}.\Qed
\end{rem}}

\subsection{Zubov's Dual Equation on a Compact Subdomain}

To facilitate data-driven techniques and prevent significant under-approximation of the ROA, we directly choose a sufficiently large compact region of interest $\rr\subseteq\R^n$. This region can either contain the entire ROA, assuming it is bounded,  or cover a significant portion of the ROA if it is unbounded. Our
proposed method uses observable data to recover the ROA relative to  $\rr$.

To incorporate Zubov's dual equation, we need to recast the dynamics in $\rr$. We first consider a first-hitting time of $\partial\rr$ defined as follows,
    \begin{equation}
    \tau:=\tau(x)=\inf\set{t\geq 0: \phi(t, x)\in\partial \rr}, \;\; x\in \rr. 
\end{equation}

We further define stopped-flow maps so that one can observe the trajectories in $\rr$. 
\begin{deff}\label{def: stopped_flow}
    Given the compact region of interest $\rr$, for each $x\in\rr$, we define the stopped-flow maps $\hphi: [0,\infty)\times \rr\ra \rr$ as
    \begin{equation}
        \hphi (t, x):=\phi(t\cj\tau, x). 
    \end{equation}
\end{deff}

Note that in the above definition, the stopping time $\tau$ implicitly encodes the information of the starting position. It can also be verified that 
\begin{enumerate}
    \item $\hphi(0, x)= x$ and $\hphi(s, \phi(t,x))=\hphi(t+s, x)$ for all $x\in\rr$,  and
    \item $\partial_t(\hphi(t, x)) = f(\hphi(t,x))$ for all $x\in\inte(\rr)$.
\end{enumerate}

We then consider a recast version of functions $V$ and $U_h$ (defined in \eqref{E: V} and \eqref{E: U}) accordingly. Let $\hV(x)=\int_0^\infty \eta(\hphi(t, x))dt$
for all $x\in \rr$.  For any $h\in C(\rr)$\footnote{Note that we have subtly changed the space of test functions from $C_b(\R^n)$ (in Eq.\eqref{E: U}) to $C(\rr)$.}, define
    \begin{equation}\label{E: Uhat}
    \hU_h(x)=
    \begin{cases}
        &\exp\set{-\hV(x)}h(\hphi_\infty(x)),\;\text{if}\;\hV(x)<\infty,\\
        & 0,\;\;\text{otherwise},
    \end{cases}
\end{equation}
where $\phi_\infty(x):=\lim_{t\ra\infty}\phi(t, x)$. For test function $h(x)=\mathds{1}(x)$, we denote $\hU:=\hU_\mathds{1}$ for simplicity. In this notion, it can be verified that $\hV(x)=\infty$ if and only if $\tau<\infty$ and $x\notin \roa(\A)$. 

The following theorem collects nice properties of $\hV$ and $\hU$ on the refined region $\rr$. 
\begin{thm}\label{thm: Zubov_dual_recast}
    For any $h\in C(\rr)$, 
    \begin{enumerate}
        \item $\hU_h$ is a viscosity
solution to 
        \begin{equation}\label{E: dual_test_hat}
    H(x, \hU_h(x), \nabla \hU_h(x))=0, \;\;\hU_h(\xeq)=h(\xeq)
\end{equation}
 on $\inte(\rr)$, where $H$ is defined in \eqref{E: H}.  Suppose that $\eta$ and $h$ are also locally Lipschitz continuous, then $\hU_h$ is the unique bounded viscosity solution. 
        \item $\hW_h=1-\hU_h$ is a viscosity solution to $Z(x, \hW_h(x), \nabla \hW_h(x))=0, \;\;\hW_h(\xeq)=1-h(\xeq)$, 
    on $\inte(\rr)$ if and only if $\hU_h$ is a viscosity solution to \eqref{E: dual_test_hat}, where $Z$ is defined in \eqref{E: Z_function}. 
        \item 
Let $F(x,  p):=-p\cdot f(x)-\eta(x).$    On  any invariant set $\I\subseteq\roa(\A)\cap \inte(\rr)$, the function $\hV(x)=-\log(\hU(x))$ is a viscosity solution to $F(x, \nabla \hV(x))=0$ with $\hV(\xeq)=0$. 
    \end{enumerate}
\end{thm}
\begin{proof}
    The proof follows a similar procedure as the proofs of the statements in Section \ref{sec: Zubov_dual}. Indeed, the proof should be the same for any $x\in\inte(\rr)$ such that $\tau=\infty$. Particularly, the dynamic programming as in the proof of Lemma \ref{prop: max_lyapunov} still holds given the flow map property of $\hphi$.  For any $x\in\inte(\rr)$ such that $\tau<\infty$, the solution is trivial considering that the quantity $\hV(x)=\int_0^\tau \eta(\hphi(s, x))ds + \eta(\phi(\tau, x))\int_\tau^\infty  ds$
    diverges. 
\end{proof}

\begin{rem}
    By Theorem \ref{thm: Zubov_dual_recast}, suppose that $\roa(\A)\cap\inte(\rr)\neq \roa(\A)$, one can only recover a portion of $\roa(\A)$ that is not absorbed by the boundary by solving \eqref{E: dual_test_hat}. This portion should be a sublevel set (relative to $\roa(\A)\cap\inte(\rr)$) of the $\hV$. 
    In view of \cite{acc2022, meng2022smooth, meng2023lyapunov}, this sublevel set is also a subset of the refined open and invariant subregion of ROA, from which trajectories will satisfy the reach-avoid-stay property. \Qed
\end{rem}

\section{Zubov-Koopman Operators and Semigroup Property}\label{sec: zubov-Koopman}

Addressing the problem of estimating the ROA, we have introduced Zubov's dual equation as well as its refined form on $\rr$. The solution involves an improper integral up to $\infty$, which requires nearly the full knowledge of the trajectory \cite{kang2021data}. To reduce the substantial amount of observation data, in this section, we derive an approximation approach using a time series. Specifically, this time series is governed by a convergent, time-homogeneous, and Feynman-Kac like semigroup,  which allows us to approximate the long-term behavior through a simple iterative process. 

We first work on $\R^n$ and then on the refined region $\rr$.
\subsection{Introducing Zubov-Koopman Operators}
 Consider 
\begin{equation}\label{E: integral}
    v_t(x):=\int_0^t\eta(\phi(r, x))dr
\end{equation}
and, for any $h\in C_b(\R^n)$ and $t>0$, we define $\T_t: C_b(\R^n)\ra C_b(\R^n)$ as
\begin{equation}\label{E: zubov_koopman}
    \T_th(x):=\exp\left\{-v_t(x)\right\}h(\phi(t, x)).
\end{equation}

The following proposition shows the basic properties of $\{\T_t\}_{t\geq 0}$. We complete the proof in Appendix \ref{sec: app2}. 
\begin{prop}\label{prop: semigroup_eigenfunction}
    $\{\T_t\}_{t\geq 0}$ is a $C_0$-semigroup. In addition, for each $t\geq 0$ and for any $h\in C_b(\R^n)$, $\T_tU_h(x)= U_h(x)$ for all $x\in\R^n$. 
\end{prop}

The stochastic version of $\set{\T_t}_{t\geq 0}$ is the famous Feynman-Kac semigroup \cite[Chapter 8]{oksendal2013stochastic}. While this is generally not true for stochastic systems, for deterministic systems, we observe that $\T_t$ depicts a form of separation of variables and can be written as a multiplication of a contraction operator with the Koopman operator, i.e. $\T_t=\exp\set{-v_t}\K_t$. 

\ymr{The following theorem also shows a close connection with the Zubov's dual equation. Particularly,  Eq. \eqref{E: feynman-kac} is a time-varying version of the Zubov's dual equation, and captures a steady point in the function space when $t\ra\infty$.} For the purpose of using the flow of $h$ governed by $\T_t$ to approximate the solution of Zubov's dual equation, for any fixed $t$, we name $\T_t$ as the \emph{Zubov-Koopman Operator}. 

\begin{thm}\label{thm: feynman-kac}
    Let the test function be $h\in C_b^1(\R^n)$. Suppose that $\eta\in C(\R^n)$ is nonnegative. Then
    \begin{enumerate}
        \item[(1)] $u_h(t, x)=\T_th(x)$ solves the following Cauchy problem, for all $t>0$ and all $x\in\R^n$,
     \begin{equation}\label{E: feynman-kac}
     \begin{cases}
         &\partial_t u_h(t,x) = \nabla_x u_h(t,x)\cdot f(x) -\eta(x)u_h(t,x), \\
         & u_h(0, x) = h(x).  
     \end{cases}
     \end{equation}
     \item[(2)] Conversely, for any $u_h\in C^{1,1}([0, \infty), \R^n)$ that satisfies \eqref{E: feynman-kac}, the solution should be of the form $u_h(t,x)=\T_th(x)$. 
    \end{enumerate}
\end{thm}
\begin{proof}
    Note that the (infinitesimal) generator of the Koopman semigroup $\{\K_t\}$ acting on $u_h$ is such that
    $$\L_f u_h(t,x) = \nabla_x u_h(t,x)\cdot f(x). $$
    To prove (1), we first look at how $u_h$ evolves according to $\K_s$ for some small $s>0$. By the definition of $\K_s$, for any fixed $t>0$, we have
    \begin{small}
     \begin{equation}\label{E: evo_K}
        \begin{split}
            &\frac{\K_su_h(t,x)-u_h(t,x)}{s} \\
             =& \frac{1}{s}\left[\exp\set{-\int_0^t\eta(\phi(r, \phi(s, x)))dr}h(\phi(t, \phi(s, x)))- u_h(t,x)\right]\\
             = &\frac{1}{s}\left[\exp\set{-\int_0^t\eta(\phi(r+s, x))dr}h(\phi (t+s, x)) - u_h(t,x)\right]\\
             = &\frac{1}{s}\left[\exp\set{-\int_s^{t+s}\eta(\phi(\sigma, x))d\sigma}h(\phi (t+s, x)) - u_h(t,x)\right]\\
        \end{split}
    \end{equation}       
    \end{small}

\noindent Note that
\begin{small}
  \begin{equation}\label{E: exp_expansion}
\begin{split}
        &\exp\set{-\int_s^{t+s}\eta(\phi(\sigma, x))d\sigma} \\
        = &\exp\set{-\int_0^{t+s}\eta(\phi(\sigma, x))d\sigma}\exp\set{\int_0^{s}\eta(\phi(\sigma, x))d\sigma}.
\end{split}
\end{equation}  
\end{small}

\noindent Combining \eqref{E: evo_K} and \eqref{E: exp_expansion}, we have
\begin{small}
 \begin{equation}
    \begin{split}
      &\frac{\K_su_h(t,x)-u_h(t,x)}{s}  \\
      =& u_h(t+s, x)\cdot\frac{1}{s}\left[\exp\set{\int_0^{s}\eta(\phi (r, x))dr}-1\right]  \\
      & +\frac{1}{s}\left[u_h(t+s, x)-u_h(t,x)\right]\\
    \end{split}
\end{equation}   
\end{small}

\noindent Sending $s\downarrow 0$ on both sides, it follows that
    \begin{equation*}
\begin{split}
        \nabla_x u_h(t,x)\cdot f(x) & = u_h(t,x)\eta(\phi_0(x))+\partial_t u_h(t,x)\\
        & = \partial_t u_h(t,x)+\eta(x)u_h(t,x),
\end{split}
\end{equation*}
which completes the first part of the proof. 

To prove (2), we suppose that $u_h\in C^{1,1}([0,\infty), \R^n)$ solves \eqref{E: feynman-kac}, then 
$$\partial_t u_h- \nabla_x u_h\cdot f+\eta u_h=0, \;\;\forall t>0, \;x\in\R^n$$
and $u_h(0, x)=h(x)$ for all $x\in\R^n$. Introduce an auxiliary function $\Psi(t, x, v)=\exp\set{-v}u_h(t,x)$. Then it is clear that $u_h(s, x)= \Psi(s, x, 0)$.  However, 
\begin{small}
    \begin{equation*}
    \begin{split}
        & d\psi(s-t, \phi(t, x), v_t(x)) \\
         = & -\partial_t\psi(s-t, \phi (t, x), v_t(x))+ \L_f \Psi(s-t, \phi(t, x), v_t(x)) \\
         &+ \eta(\phi (t, x))\partial_v \Psi(s-t, \phi (t, x), v_t(x))\\
          =&\exp\set{-v_t(x)}[-\partial_tu_h(s-t, \phi (t, x))+\L_f u_h(s-t, \phi (t, x)) \\
          &+ \eta(\phi (t, x))u_h(s-t, \phi (t, x))] = 0, 
    \end{split}
\end{equation*}
\end{small}

\noindent where the quantity $v_t(x)$
is defined as in \eqref{E: integral}.  
Therefore, the quantity $\Psi(s-t, \phi (t, x), v_t(x))$ is a constant for all $t$ and $x$, which implies that
\begin{small}
   \begin{equation*}
    \begin{split}
        u_h(s, x)&=\Psi(s, x, 0)\\
        & =\Psi(s, \phi (0, x), v_0(x))=\Psi(0, \phi (s, x), v_s(x))\\
        & = \exp\{-v_s(x)\}u_h(0, \phi (s, x)) = \exp\{-v_s(x)\}h(\phi (s, x)).
    \end{split}
\end{equation*} 
\end{small}

\noindent The proof is completed. 
\end{proof}

\subsection{A Time-Series Approximation}\label{sec: time-series}
\ymr{Clearly, by Theorem \ref{thm: feynman-kac} and by the definition of $U_h$ in \eqref{E: U}, for any $h\in C_b^1(\R^n)$, we have that $\lim_{t\ra\infty}u_h(t , x)=\lim_{t\ra\infty}\T_th(x)=U_h(x)$ for all $x\in\R^n$.}  In particular, $\lim_{t\ra\infty}\T_t\mathds{1}=U$ uniformly, and $U$ is also the unique, up to multiplicative constants,  fixed point of $\{\T_t\}_{t\geq 0}$. To approximate $U$, one  can pick a fixed time interval $\Delta t$, and define 
\begin{small}
    \begin{equation}\label{E: T_delta}
    \T_\Delta:=\T_{\Delta t}
\end{equation}
\end{small}

\noindent as well as 
    \begin{equation} \label{E: k_iteration}
    \T_{k\Delta}:=\underbrace{\T_\Delta\circ\T_\Delta\cdots \circ\T_\Delta}_{k\; \text{iterations}}. 
\end{equation}
 Then, by \eqref{E: U} and the uniqueness (up to multiplicative constants) of the fixed point, the composed operator $ \T_{k\Delta}: C_b(\R^n)\ra C_b(\R^n)$ for any $k\geq 1$,  and  $\lim_{k\ra\infty}\T_{k\Delta}h =U$ for any $h$ such that $h(\xeq)=1$. Suppose that one can approximate $\T_\Delta$ properly, then $\T_{k\Delta}h$ for some large $k$ should be a reasonably good approximate for $U$. 

Similar to the approximation of Koopman operators, to obtain a discrete version $\Tb$ of the bounded linear  operator $\T_\Delta$,  it usually relies on the choice of a (discrete) dictionary of observable test functions, denoted by 
\begin{small}
    \begin{equation}\label{E: dict}
    \Zk_N(x):=\left[\zk_0(x),\zk_1(x),\cdots,\zk_{N-1}(x)\right], \;\;N\in\N\cup\{\infty\}. 
\end{equation}
\end{small}

\noindent Then, the approximation $\tT: \operatorname{span}\{\zk_i\}_{i=0}^{N-1}\ra \operatorname{span}\{\zk_i\}_{i=0}^{N-1}$ is valid in the sense that, for each $h\in C_b(\R^n)$,  there exists an $\hk\in \operatorname{span}\{\zk_i\}_{i=0}^{N-1}$ and a uniformly continuous residual term  $\tta\in C_b(\R^n)$ such that
\begin{small}
    \begin{equation}\label{E: approx_T}
\T_\Delta h = \tT \hk + \tta.
\end{equation}
\end{small}

Possible choices of the dictionary $\Zk_N$ have been discussed in \cite{williams2015data, deka2022koopman},  including polynomials, Fourier basis, spectral elements, and neural network-based functions.  These choices are generally locally Lipschitz continuous, but there may be cases where differentiability is not exhibited, as seen in neural network-based functions with ReLU as the activation function. 

In line with Theorem \ref{thm: feynman-kac} and viscosity regularity of Zubov's dual equation, to make the time series approximation of $U$ robust, we extend Theorem \ref{thm: feynman-kac} for test functions $h\in C_b(\R^n)\cap\loclip(\R^n)$ and verify the regularity. 

\begin{thm}\label{thm: feynman-kac_extension}
    Let the test function be $h\in C_b(\R^n)\cap\loclip(\R^n)$. Suppose that $\eta\in C(\R^n)$ is nonnegative and locally Lipschitz. 
    Then, for each $t>0$, $u_h:=\T_th$ is the unique viscosity solution to \eqref{E: feynman-kac}. 
    
    Furthermore, for each fixed $t>0$, suppose that there exists a family of functions $\set{\uk^t_k\in C_b(\R^n)\cap\loclip(\R^n)}_{k=0}^\infty$ and, accordingly, a family of uniformly bounded continuous residuals $\set{\tta^t_k}_{k=0}^\infty$ with Lipschtiz constant $L_k$, such that 
    \begin{equation*}
        u_h(t,x):=\T_th(x)=\uk^t_k(x)+\tta^t_k(x), \;\;x\in\R^n
    \end{equation*}
and $\tta^t_k(\xeq)=0$. 
Assume $\sup_{x\in\Omega}|\tta_k^t(x)|\ra 0$   on each bounded subdomain $\Omega\subseteq\R^n$. Then, as $t,  k\ra \infty$, $\uk_k^t$ converges uniformly to $U_h$. 

In addition, suppose we also have $L_k\ra 0$, then the function $\vk_k^t=-\log(\uk_k^t)$ is the unique viscosity solution to $-\nabla \vk_k^t(x)\cdot f(x)-\eta(x)+\mathcal{O}_k(x)=0$ with $\vk_k^t(\xeq)=0$, where $\sup_{x\in\Omega}|\mathcal{O}_k(x)|\ra 0$  (as $k\ra\infty$) on  each bounded subdomain $\Omega\subseteq\R^n$.  

\begin{proof}
Since the proof of viscosity property is similar to the proof for $U_h$, we omit the first part of the proof. 

    Now, as $t\ra \infty$, by the uniform convergence of $\T_th$ for any $h\in C_b(\R^n)$, one has $|\partial_tu_h(t,x)|$ uniformly converges to $0$ on each subdomain $\Omega\subseteq\R^n$. By the construction of $\uk_k^t$, as $k\ra \infty$, $\uk_k^t$ also uniformly converges to $u_h(t, \cdot)$. Therefore, for each $h$, $\uk_k^t\ra U_h$ uniformly, which completes the second part of the proof. 

    By the construction of $\uk_k^t$, and by the first part of Theorem \ref{thm: feynman-kac_extension}, we can immediately have that, for each $t$ and for each $k$, $\uk_k^t$ is the unique viscosity solution to
    \begin{small}
        \begin{equation}\label{E: u_satisfaction}
\begin{split}
        &\nabla  \uk_k^t(x)\cdot f(x)  \\
        = &\partial_t u_h(t,x)+\eta(x)\uk_k^t(x) + \eta(x)\tta_k^t(x)- \nabla  \tta_k^t(x)\cdot f(x). 
\end{split}
\end{equation}
    \end{small}

\noindent Since $\tta_k^t$ is necessarily locally Lipschitz continuous, $\nabla \tta_k^t(x)$ exists a.e.. At a differentiable point, we have
        \begin{equation*}
    \begin{split}
        |\nabla \tta_k^t(x)|=\left|\frac{\tta_k(x+h)-\tta_k(x)}{h}\right|\leq L_k \ra 0.
    \end{split}
\end{equation*}
 The above implies that  $|\nabla \tta_k^t|$ converges uniformly to $0$ a.e.. For sufficiently large $t$ and $k$, by a similar argument as  the first part\footnote{One can take a convolution on both side of \eqref{E: u_satisfaction} with a mollifier $\rho_\eps$ and let $\widetilde{\mathcal{O}}_k=\left(\eta\tta_k^t- \nabla  \tta_k^t\cdot f\right)*\rho_\eps$. Then $\widetilde{\mathcal{O}}_k$ satisfies the requirement. }, one can verify that there exists an $\widetilde{\mathcal{O}}_k$ such that $\sup_{x\in\Omega}|\widetilde{\mathcal{O}}_k(x)| \ra 0$ 
and $\uk_k^t$ is the viscosity solution to 
\begin{equation}\label{E: lip_conv}
    -\nabla  \uk_k^t(x)\cdot f(x)  
        = -\eta(x)\uk_k^t(x) +\widetilde{\mathcal{O}}_k(x). 
\end{equation}
By a similar argument as in Proposition \ref{prop: max_lyapunov}, given the smoothness of $-\log(\uk_k^t)$,  there exists an $\mathcal{O}_k$ with $\sup_{x\in\Omega}|\mathcal{O}_k(x)| \ra 0$  such that  $\vk_k^t$ is  the unique viscosity solution to $-\nabla \vk_k^t(x)\cdot f(x)-\eta(x)+\mathcal{O}_k(x)=0$ with $\vk_k^t(\xeq)=0$. 
\end{proof}

\begin{rem}\label{rem: lya-converge}
    The first two parts of the above theorem state that, even though we may only use a family of uniformly convergent locally Lipschitz functions to approximate, the limit still satisfies \eqref{E: feynman-kac} \ymr{in} a viscosity sense. In addition, we can find a proper approximation $\uk_k^t$ for $U_h$. \ym{For the purpose of approximating $\{x\in\R^n: U(x)\geq 0\}$ w.r.t. the Hausdorff metric\footnote{\ym{For any approximator $H$ of $U$ such that $\|U-H\|\leq \eps$, one can show that the set $\{x: H(x)\geq \eps\}$ is a tight inner approximation of the ROA.  By definition, the Hausdorff distance between the ROA  and $\{x: H(x)\geq \eps\}$ is bounded by a small value determined by $\eps$, reflecting the proximity of the two sets in terms of their level set definitions.  
    However, there is a possibility that the approximated boundary `crosses' the $\partial\D(\A)$. To determine whether it provides a proper inner approximation of the ROA, one must undergo formal verifications by checking the Lie derivative properties, as stated in Section VI.C and detailed in the reference \cite{liu2023physics,liu2023towards}.}}, this approximation is satisfactory.}

    On the other hand, without the assumption that $L_k\ra 0$, the approximation $\uk_k^t$ may not solve \eqref{E: u_satisfaction} with vanishing $\nabla\tta_k^t\cdot f$ in a proper sense. This may eventually cause the third part of the statement to fail to hold, meaning that the approximation $\vk_k^t$ can not be readily used as a Lyapunov function. We will also demonstrate this effect in Section \ref{sec: num} via examples.\Qed
\end{rem}
\end{thm}

At the end of this section, we make a quick extension of the aforementioned results, further refining our observations on the compact region of interest $\rr$. Due to the similarity with previous results, we omit the proof for the following Corollary.

\begin{cor}\label{cor: main}
    Recall the stopped-flow map $\hphi$ defined in Definition \ref{def: stopped_flow}. Let 
    $\hv_t(x):=\int_0^t\eta(\hphi(r, x))dr$. For any $h\in C(\rr)$ and $t>0$, we redefine $\T_t: C(\rr)\ra C(\rr)$ as
        \begin{equation}\label{E: zubov_koopman_redefined}
    \T_th(x):=\exp\left\{-\hv_t(x)\right\}h(\hphi(t, x)).
\end{equation}
Then, 
\begin{enumerate}
\item $\set{\T_t}_{t\geq 0}$ is a $C_0$-semigroup. 
\item For each each $h\in C(\rr)$ and for each $t$, $\hU_h$ is an eigenfunction such that $\T_t\hU_h=\hU_h$.  
    \item For any test function $h\in \lip(\rr)$, given that $\eta\in\lip(\rr)$ is nonnegative, then $\hu_h(t, x):=\T_th(x)$ is the unique viscosity solution to \eqref{E: feynman-kac} for all $t>0$ and $x\in\inte(\rr)$. 
    \item Suppose there exists a family of functions $\set{\uk^t_k\in \lip(\rr)}_{k=0}^\infty$ that uniformly converges to $\hu_h(t, \cdot)$ for any $t$, then, as $t,k\ra\infty$, $\uk_k^t$ converges uniformly to $\hU_h$, where $\hU_h$ is defined in \eqref{E: Uhat}. 
    \item Suppose we also have $L_k\ra 0$, which is the Lipschitz constant for $\hu_h(t,\cdot)-\uk_k^t$ for each $t$. Then, for sufficiently large $t$ and $k$, the function $\vk_k^t=-\log(u_k^t)$ is the unique viscosity solution to $-\nabla \vk_k^t(x)\cdot f(x)-\eta(x)+\mathcal{O}(x)=0$ with $\vk_k^t(\xeq)=0$  on any invariant set $\I\subseteq \roa(\A)\cap\inte(\rr)$ , where $\sup_{x\in\I}|\mathcal{O}(x)|$ is arbitrarily small.  
\end{enumerate}
\end{cor}

\begin{rem}
    We have chosen not to introduce the `hat' notation for the redefined $\T_t$ in \eqref{E: zubov_koopman_redefined}. \ymr{The difference between the two versions (\eqref{E: zubov_koopman} and \eqref{E: zubov_koopman_redefined}) lies in their domains. We encourage readers to verify the domain of $\T_t$ in the context before using it.}  \Qed
\end{rem}

\section{Finite-
Dimensional Approximation of Zubov-Koopman operators.}\label{sec: finite-approx}

As we have seen in \eqref{E: approx_T}, we expect to find an approximation 
for the Zubov-Koopman operators such that the image functions converge uniformly.  In practice, we would like to see if the training discrete dictionary $\Zk_N$ (as in \eqref{E: dict}) of observable test functions can be reduced to finite, such that the approximation 
behaves like a finite-rank operator, and preserves $\operatorname{span}\set{\zk_i}_{i=0}^{N-1}$ for some $N<\infty$. 

 In this section, we rigorously investigate basic properties and a finite-dimensional approximation of Zubov-Koopman operators. 
 We will work on the compact region of interest $\rr\subseteq\R^n$ for the rest of this paper. 
Since our purpose is to learn  Zubov-Koopman operators based on training data, we propose a three-step intermediate approximation for $\{\T_t\}$, such that for any $t>0$, an approximation of the form \eqref{E: approx_T} holds. 

\subsection{Compact Approximation of Zubov-Koopman}\label{sec: compact_operator}

 $\{\T_t\}_{t\geq 0}$ defined in \eqref{E: zubov_koopman_redefined} is clearly a family of bounded linear operators. One may attempt to show that $\T_t$ is also compact for each $t$. It suffices to show that $\T_t(\B_r) \subseteq C(\rr)$ is relatively compact, where $\B_r = \{h\in C(\rr): \|h\|_\infty\leq r\}$ for some $r>0$. However, equicontinuity within $\T_t(\B_r)$ is not guaranteed. To see this, we set $h_n(x) = \sin(nx) \in \B_1$ (or similarly, the Fourier basis), and let $\hphi(t, x)=x\cdot e^{-t}$ for all $x\in\inte(\rr)$ and $\hphi(t, x)=x$ elsewhere. 
Then, the sequence $\{h_n\circ \hphi(t,\cdot)\}_n$ for each $t$ does not possess equicontinuity due to the rapid oscillation as $n$ increases. 

Nonetheless, the following proposition states that one can use compact operators to strongly approximate $\T_t$ for each $t$. 

\begin{prop}\label{prop: operator_approx}
For each $t>0$, there exists a family of compact linear operator $\{\T_t^\eps\}_{\eps>0}$, such that for all $h\in C(\rr)$, we have
$\|\T_t^\eps h - \T_t h\|_\infty\ra 0$ as $\eps\ra 0$. 
\end{prop}

\begin{proof}
For each $t$, we rewrite the Koopman operator as $ \T_t h(x)  
         = \int_\rr \delta(x-y) \T_th(y) dy$, 
where $\delta$ is the Dirac delta. Now we use a family of integral operators with smooth kernels to approximate the above distribution. The idea is to approximate the Dirac delta using a smooth mollifier $\rho\in C_0^\infty(\rr)$ as we have seen in the proof of Theorem \ref{thm: feynman-kac_extension}. For each $\eps>0$, let 
$\rho_\eps(x):= \frac{1}{\eps^n}\zeta\left(\frac{x}{\eps}\right)$ and $\int_\rr \rho(y)dy = 1$. It can be verified that $\rho_\eps \subset C_0^\infty(\rr)$ with a compact support in $\B_\eps$. We define the approximation as
$\T^\eps_t h(x) = \int_\rr \rho_\eps(x-y) \T_th(y) dy$ for all $x\in\rr$. 
It is a well-known result that, for each $t$ and $\eps>0$, the operator $\T_t^\eps: C(\rr)\ra C(\rr)$ is compact given its smooth kernel. 

To verify the convergence property, the rest of the proof falls in a standard procedure. 
Now, for each $t$ and all $x\in\rr$, by change of variable, we have that $\T_t^\eps h(x)= \int_{\B_1} \rho(y) \T_th(x-\eps y) dy$. 
It follows that
\begin{small}
    \begin{equation}\label{E: bound}
    \begin{split}
        |\T_t^\eps h(x) - \T_t h(x)|
        \leq & \left|\int_\rr \rho_\eps(x-y) [\T_th(y) - \T_th(x)]dy\right|\\
        \leq & \int_{\B_1} \rho(y) |\T_th(x-\eps y)-\T_th(x)|dy\\
         \leq & \sup_{y\in\B_1} |\T_th(x-\eps y)-\T_th(x)|.
    \end{split}
\end{equation}
\end{small}

\noindent Note that, tor each $t$, given any $h\in C(\rr)$ and any Lipshitz continuous map $\hphi(t, \cdot): \rr\ra\rr $, the composition  $h\circ\hphi$ is a uniform continuous function on $\rr$. Therefore, by definition, $\T_th$ is also uniformly continuous on $\rr$.  The last term above converges to $0$ as $\eps\downarrow 0$ for all $x\in\rr$. Therefore, for each $t$, taking the supremum on both sides of  \eqref{E: bound} and sending $\eps$ to $0$, we have $\|\T_t^\eps h - \T_t h\|_\infty\ra 0$. 
\end{proof}

\begin{cor}
    For each $t, s>0$, there exists a family of compact linear operator $\{\T_t^\eps\}_{\eps>0}$, such that for all $h\in C(\rr)$, we have
$\|(\T_t^\eps \circ\T_s^\eps)  h - (\T_t \circ \T_s)h\|_\infty\ra 0$ as $\eps\ra 0$.
\end{cor}
\begin{proof}
    We can use the same $\{\T_t^\eps\}_{\eps>0}$ as in Proposition \ref{prop: operator_approx}. Then the composition $(\T_t^\eps \circ\T_t^\eps)h=(\rho_\eps*\rho_\eps)* \T_{t+s}h$, where $(\rho_\eps*\rho_\eps)$ is a smoother mollifier than $\rho_\eps$ with the same convergence property relative to the Dirac Delta. The rest of the proof should be the same as in Proposition \ref{prop: operator_approx}. 
\end{proof}

\begin{rem}
The convergence in Proposition \ref{prop: operator_approx} cannot be extended to the convergence w.r.t. the operator norm, i.e. $\|\T^\eps_t-\T_t\| := \sup_{\|h\|_\infty=1}\|(\T^\eps_t-\T_t)h\|\ra 0$. One may revisit the example  $h_n(x)=sin(nx)$ for $n\in\mathbb{N}$. It is clear that $\|h_n\|_\infty = 1$ for all $n$, but based on the inequality in \eqref{E: bound}, the uniform limit fails to exist due to the unbounded Lipschitz constants of $\{h_n\}$. However, for the purpose of this section, the convergence in Proposition \ref{prop: operator_approx} is already satisfactory. \Qed
\end{rem}

\subsection{Finite Dimensional Representation}\label{sec: finite_approx}
Since the kernels of $\{\T^\eps_t\}$ are smooth and  compactly supported, it is a well-known result that the operators within this family are Hilbert-Schmidt operators. 

We then investigate the spectral behavior of $\T_t^\eps$ within a separable Hilbert space $\H$ with the  inner product $\langle\cdot,\cdot\rangle:= \langle\cdot,\cdot\rangle_{L_2}$.  
Let $\{\zeta_i\}_{i\in \Z} \subset\H$ be the  eigenfunctions of $\T_t^\eps$. 
Then, for all $h\in C(\rr)\subset\H$, we have that  
$\T^\eps_t h = \sum_{i=0}^\infty e^{\lambda_i^\eps t} \langle h, \zeta_i\rangle \zeta_i$
as well as $\T^\eps_t \zeta_j= \sum_{i=0}^\infty e^{\lambda_i^\eps t} \langle \zeta_j, \zeta_i\rangle \zeta_i=e^{\lambda_j^\eps t}  \zeta_j$
where $\{\lambda_i^\eps\}$ are the corresponding eigenvalues in the $\log$-scale. 
\begin{rem}
In the above infinite-sum representations, we have implicitly assumed that the eigenfunctions are real-valued. However, this is not always the case. For more general situations, $\langle \zeta_i, \bar{\zeta}_i\rangle =1$ for all $i$, and $\T^\eps_t h = \sum_{i\in\Z} e^{\lambda_i^\eps t} \langle h, \bzeta_i\rangle \zeta_i$
as well as $\T^\eps_t \zeta_j= \sum_{i\in\Z} e^{\lambda_i^\eps t} \langle \zeta_j, \bzeta_i\rangle \zeta_i=e^{\lambda_j^\eps t}  \zeta_j$.  For simplicity, we still use the above expressions without further indicating whether they are real or complex-valued.
\Qed
\end{rem}

\begin{prop}\label{prop: finite_1}
For any fixed $t>0$, for any arbitrarily small $\tta>0$, there exists a sufficiently large $N$ and a finite-dimensional approximation \ymr{$\T_{t,N}^\eps$} such that 
$\|\T_{t,N}^\eps h-\T_t h\|_\infty<\tta, \;\;h\in C(\rr).$
\end{prop}
\begin{proof}
The Hilbert-Schmidt operator $\T^\eps_t$ is also necessarily a finite-rank operator, i.e., for any complete orthonormal basis $\{g_j\}\subset \H$, $\sum_{j=0}^\infty\langle \T^\eps_t g_j, g_j\rangle  = \sum_{j=0}^\infty\sum_{i=0}^\infty e^{\lambda_i^\eps t} |\langle g_j, \zeta_i\rangle|^2= \sum_{j=0}^\infty e^{\lambda_j^\eps  t}<\infty$. 
Denoting the finite truncation as $\T_{t,N}^\eps:=\sum_{i=0}^{N-1} e^{\lambda_i t} \langle\; \cdot\;, \zeta_i\rangle \zeta_i$,  we have $\lim\limits_{N\ra\infty}\|\T_{t,N}^\eps h-\T^\eps_t h\|_\infty=0$ for all $h\in C(\rr)$.
Combining Proposition \ref{prop: operator_approx}, we have the desired property. 
\end{proof}

\begin{rem}\label{rem: extension}
In contrast to the approach described in \cite[Section III, 2)]{mauroy2019koopman}, wherein the finite-rank operator is limited to preserving only a finite-dimensional subspace of the initial function space, our proposed finite-rank approximate operator, as in Proposition \ref{prop: finite_1}, demonstrates better theoretical robustness. Notably, it retains a valid infinite-dimensional domain, i.e. $h\in C(\rr)$, identical to that of
$\T_t$ for all $t$.

Particularly, suppose $\{\zeta_i\}$ are complex-valued and $U$ is the solution to \eqref{E: dual}, by the eigenvalue problem in Proposition \ref{prop: semigroup_eigenfunction}, we have that $ U = \T_tU \approx \T_t^\eps U = \sum_{i\in\Z} e^{\lambda_i^\eps t} \langle U, \bzeta_i\rangle \zeta_i\approx \sum_{i=-N}^N e^{\lambda_i^\eps t} \langle U, \bzeta_i\rangle \zeta_i$ and $\langle U, \bzeta_i\rangle =  \langle \T_tU, \bzeta_i\rangle \approx e^{\lambda_i^\eps t}\langle U, \bzeta_i\rangle$  for some $i$ such that $\langle U, \bzeta_i\rangle\neq 0$. Then, the corresponding $e^{\lambda_i^\eps t}\in\R$ is an approximation of the eigenvalue $1$ of  $\T_t$, but the eigenfunction $\zeta_i$ is not a direct approximation for $U$. We still need the expansion form to properly approximate $U$. 
 \Qed
\end{rem}

\subsection{Feasibility of Data-driven Approximation}\label{sec: data_approx}
For any fixed $\Delta t$, based on Proposition \ref{prop: finite_1}, the original operator $\T_\Delta$ (recall Definition \ref{E: T_delta} for $\T_\Delta$)  can be approximated (strongly)  with an arbitrarily high precision by a finite-dimensional bounded linear operator \ymr{$\T^\eps_{\Delta,N}$}, which preserves the non-ignorable eigenmodes. We would like to further use a data-driven approach to learn the eigenvalues and the eigenfunctions for this finite-dimensional operator.

To do this,  we consider a sufficiently dense $N$-dimensional subspace $\mathcal{F}_N\subset C(\rr)$ and choose observable test functions 
$\Zk_N(x)=\left[\zk_0(x),\zk_1(x),\cdots,\zk_{N-1}(x)\right]$ from $\F_N$.  Then, 
$\T^\eps_{\Delta,N} \Zk_N(\cdot) = \sum_{i=0}^{N-1} e^{\lambda_i^\eps t}  \zeta_i(\cdot) \pk_i \xrightarrow{\|\cdot\|_\infty} \sum_{i=0}^\infty e^{\lambda_i^\eps t}  \zeta_i(\cdot) \pk_i$,
where $\pk_i=[\langle \zk_0, \zeta_i\rangle,\; \langle \zk_1, \zeta_i\rangle, \cdots \langle \zk_{N-1}, \zeta_i\rangle]$ represents the vector of the projection inner products. If $\operatorname{span}\{\zk_i\}_{i=0}^\infty$ is also dense in $C(\rr)$, then this finite-dimensional approximation 
is of the desired form as in \eqref{E: approx_T}. 

Suppose that $\F_N$ is  given arbitrarily,  then each $\zk_i$ 
may not be invariant within $\{\zeta_i\}_{i=0}^{N-1}$  
under $\T_\Delta$. However, 
when restricted on the $N$-dimensional truncation, the approximate operator $\T^\eps_{\Delta,N}$ preserves the space $\{\zeta_i\}_{i=0}^{N-1}$. 
We aim to express $\{\zeta_i\}_{i=0}^{N-1}$ by functions from $\operatorname{span}\{\zk_0, \zk_1, \cdots,\zk_{N-1}\}$. 

Taking advantage of the linear transformation, for systems with unknown dynamics, data-driven methods will approximate
the Zubov-Koopman eigenvalues and eigenfunctions by fitting the data 
$\Zk_N(x^{(m)})$ and $\T_{\Delta}\Zk_N(x^{(m)})$  for some $\{x^{(m)}\}_{m=0}^{M-1}\subset\rr$. We also denote by $\mathfrak{X}, \mathfrak{Y}\in\C^{M\times N}$  the stack of input and output data. 
The best-fit matrix $\Tb$ is formulated as an optimization problem $\Tb = \operatorname{argmin}_{A\in \C^{N\times N}}\|\mathfrak{Y}-\mathfrak{X}A\|_F.$

\ym{
The learned operator has the following sense of approximations for sufficiently large $N$ and $M$. Due to its substantial similarity to EDMD \cite{williams2015data, Brunton2022survey}, we omit the proof for this part of the convergence.}
\begin{enumerate}
\item Let $(\mu_i, \vb_i)_{i=0}^{N}$ be the eigenvalues and eigenvectors of $\Tb$. Then, for each $i$, the learned eigenvalues $\tl_i$  is such that $\lambda_i^\eps\approx\tl_i=\log(\mu_i)/t$, and the learned eigenfunction $\tzeta_i$ satisfies  $\zeta_i(x)\approx \tzeta_i(x)=\Zk_N(x)\vb_i$. \footnote{The eigen-approximation is only for the compact operator $\T_t^\eps$ introduced in Section \ref{sec: compact_operator}. The original operator $\T_t$ may not possess non-trivial point spectrum other than $1$. }
    \item For any $\hk\in\operatorname{span}\{\zk_0, \zk_1, \cdots,\zk_{N-1}\}$ such that $\hk(x)=\Zk_N(x)\mathbf{w}$ for some column vector $\mathbf{w}$, we have that 
    \begin{small}
            \begin{equation}\label{E: finite_expansion_approx}
        \T_t\hk(\cdot)\approx\T_{t,N}^\eps\hk(\cdot)\approx \Zk_N(\cdot)(\Tb\mathbf{w}). 
    \end{equation}
    \end{small}
\end{enumerate}

By the well-known universal approximation theorem, all of the function approximations from above should have uniform convergence. 

\section{Data-driven Algorithms}\label{sec: algrithms}

By the finite rank expansion \eqref{E: finite_expansion_approx}, one can easily verify that
    \begin{equation}\label{E: approx_U}
    \T_{k\Delta}\hk(\cdot)\approx \Zk_N(\cdot)(\Tb^k\mathbf{w}) = \sum_{i=0}^{N-1}e^{\tl_ik\Delta t}\tzeta_i(\cdot)(\vb_i\cdot\mathbf{w}), 
\end{equation}
where $\mathbf{w}$ is such that $h(\cdot)=\Zk_N(\cdot)\mathbf{w}$. \ym{Recall the beginning part of Section IV. B, particularly Eq. \eqref{E: k_iteration} and $\lim_{k\ra\infty}\T_{k\Delta}h =U$ for any $h$ such that $h(\xeq)=1$. Eq. \eqref{E: approx_U} is then a numerical version for approximating $U$ using iterations.} By virtue of Proposition \ref{prop: semigroup_eigenfunction} and Corollary \ref{cor: main}, for all $t$, the real-valued eigenfunction of $\T_t$ with the corresponding eigenvalue $1$ should be $U$, up to multiplicative constants.  Similarly, the only non-vanishing mode(s) as $k\ra\infty$ should be the ones with $\tl_i=0$, which will be readily a solution to the Zubov's dual equation suppose it is real-valued. If this happens, we can directly choose  $\tzeta_i$ associated with $\tl_i=0$ as the approximation for $U$ without any iteration. However, due to the numerical error and the possibly complex-valued eigenvectors of $\Tb$, 
one still needs to use the iteration form for some sufficiently large $k$ as the final approximation. 

Nonetheless, the key is to learn $\T_\Delta$ using observable data of trajectories within a relatively shorter time period, and obtain the matrix $\Tb$ so that the  \eqref{E: approx_U} can be utilized. To do this, we modify the existing Koopman learning techniques for Zubov-Koopman operators $\T_\Delta$, as defined in \eqref{E: T_delta}, with a fixed training time interval $\Delta t$. 


\subsection{Generating Training Data}
As we have briefly discussed about the feasibility of using data-driven approaches to approximate the operator $\mathcal{T}_\Delta$ for each $\zk\in\Zk_N$ and each $x\in\rr$, we consider $\zk(x)$ as the features and $\T_\Delta\zk(x):=\exp\left(-\int_0^{\Delta t}\eta(\hphi(s, x))ds\right)\zk(\hphi(\Delta t,x))$ as the labels. We inevitably need to acquire the stopped-flow map $\hphi$ and compute the integral. Drawing inspiration from \cite{kang2021data}, for each termination time, we can assess both the trajectory (without stopping) and the integral, i.e. the pair $\left(\phi(t,x), \int_0^t\eta(\phi(s, x))ds\right)$,  by solving a single augmented ODE system 
    \begin{equation}\label{E: augmented}
    \begin{split}
        &\dot{\xb}(t)=f(\xb(t)), \;\;\xb(0)=x\in\R^n,\\
        &\dot{I}(t) = \eta(\xb),\;\;I(0)=0.
    \end{split}
\end{equation}
 We propose the following algorithm based on this technique for our purpose of evaluating $\T_\Delta \zk(x)$ for each $\zk$ and each $x$. 
\begin{alg}\label{alg: data_generate}
Let the region of interest $\rr$, the dictionary $\Zk_N(\cdot)=\left[\zk_0(\cdot),\zk_1(\cdot),\cdots,\zk_{N-1}(\cdot)\right]^T$, the  function $\eta\in C(\R^n)$, and the time interval $\Delta t$ be given. Then, for each $i\in\{0,1, \cdots N-1\}$  and $x\in\rr$, we compute $\T_\Delta \zk_i(x)$ as follows. 
    \begin{enumerate}
        \item (Solving \eqref{E: augmented}) Choose a number of partition points  $\Nb$. Generate a uniform partition 
        with $\Nb$ points from the interval $[0, \Delta t]$. Denote the size of each partition as $\delta t$. Solve \eqref{E: augmented} using an ODE solver with this partition.  Denote the solution arrays as $\vect(\phi,x):=[\phi(j\cdot\delta t,x),\;\text{for}\; j \in\{0, 1, \cdots, \Nb-1\}]$ 
        and $\vect(I):=[I(j\cdot\delta t),\;\text{for}\; j \in\{0, 1, \cdots, \Nb-1\}]$.
        Also, denote the $j$-th element of the vectors as $\vect_j( )$.
        \item (Identifying out-of-domain index) Identify the first element numbered as $\iota$ in $\vect(\phi, x)$ that is not within $\rr$. 
        \item (Identifying the intersection point on the boundary) If $\iota \notin\{0, \cdots, \Nb-1\}$, continue to the next step. Otherwise,   find the line segment connecting the $(\iota-1)$-th and $\iota$-th elements (i.e. $\vect_{\iota-1}(\phi, x)$ and $\vect_{\iota}(\phi, x)$). Find the intersection point $\rr_{\text{sect}}$ of the line segment and $\partial\rr$. 
        \item (Modifying the trajectory and the integral) If $\iota \notin\{0, \cdots, \Nb-1\}$, continue to the next step. Otherwise, for each $j\in\{\iota, \cdots \Nb-1\}$, replace $\vect_j(\phi, x)$ by $\rr_{\text{sect}}$, and replace $\vect_j(I)$ by $\vect_{\iota-1}(I)+ (j-\iota+1)\rr_{\text{sect}}\cdot\delta t$. Keep the notation for the modified $\vect(\phi,x)$ and $\vect(I)$. 
        \item $\T_\Delta \zk_i(x) = \exp\set{-\vect_{\Nb-1}(I)}\zk_i(\vect_{\Nb-1}(\phi,x)) $. 
    \end{enumerate}
\end{alg}

We summarize the algorithm for generating training data for one time period in Algorithm \ref{alg: one_time}. 

\setcounter{algorithm}{1}
\begin{algorithm}[H]
	\caption{Generating Training Data}\label{alg: one_time} 
	\begin{algorithmic}[1]
		\Require  $f$, $\eta$, $\rr$, $\Zk_N$, $\Delta t$, and a set  $\set{x^{(m)}}_{m=0}^{M-1}\subset \rr$.
  \For{$m$ \textbf{from} 0 \textbf{to} $M-1$}
  \For{$i$ \textbf{from} 0 \textbf{to} $N-1$}
  \State Calculate $\zk_i(x^{(m)})$
   \State Calculate $\T_\Delta\zk_i(x^{(m)})$ using Algorithm \ref{alg: data_generate}
 \EndFor
 \State Stack 
 \begin{small}
      \begin{equation*}
     \T_\Delta\Zk_N(x^{(m)})=[\T_\Delta\zk_0(x^{(m)}),  \cdots,\T_\Delta\zk_{N-1}(x^{(m)})]
 \end{equation*}
 \end{small}
 \EndFor
 \State Stack \begin{small}
     $\mathfrak{X}, \mathfrak{Y}\in\C^{M\times N }$
 \end{small} such that 
\begin{small}
    $\mathfrak{X}=[\Zk_N(x^{(0)}), \Zk_N(x^{(1)}),\cdots, \Zk_N(x^{(M-1)})]^T$
\end{small} 
 and \begin{small}
     $\mathfrak{Y}=[\T_\Delta\Zk_N(x^{(0)}), \T_\Delta\Zk_N(x^{(1)}),\cdots, \T_\Delta\Zk_N(x^{(M-1)})]^T$
 \end{small}
	\end{algorithmic}
\end{algorithm}

\subsection{EDMD Algorithm for Operator Learning}\label{sec: edmd}
After gathering training data, the rest of the training process should be the same as the existing Koopman operator learning techniques. We perform the EDMD algorithm 
to obtain the final approximation for $\T_\Delta$. 

EDMD algorithm provides an estimation of $\T_\Delta$ using one time period observation data. We use Algorithm \ref{alg: one_time} to obtain training data $(\mathfrak{X}, \mathfrak{Y})$. As we have briefly discussed in Section \ref{sec: data_approx}, we need to find $\Tb = \operatorname{argmin}_{A\in\C^{N\times N}}\|\mathfrak{Y}-\mathfrak{X}A\|_F$ in order to obtain the approximation of the form \eqref{E: finite_expansion_approx}. For EDMD, the $\Tb$ is given in closed-form as 
\begin{small}
    \begin{equation}
    \Tb = \left(\mathfrak{X}^T\mathfrak{X}\right)^\dagger\mathfrak{X}^T\mathfrak{Y},
\end{equation}
\end{small}

\noindent where $\dagger$ is the pseudo inverse. 

\subsection{Predicting ROA and Constructing Lyapunov Function}\label{sec: lya_nn}
To predict $U$, Eq~\eqref{E: approx_U} is also ready to use. We simply pick a tolerance and the largest number of iteration $K$, then iterate \begin{small}
    $\Tb$ 
\end{small} until \begin{small}
    $\|\Tb^k-\Tb^{k-1}\|_F$
\end{small}  reaches the threshold or $k=K$, whichever comes first. For simplicity, we can also choose the dictionary $\Zk_N$ to include a real-valued $\zk_i$ such that $\zk_i(\xeq)=1$. Then, we obtain a Zubov-Koopman approximation of $U$,  
\begin{small}
    \begin{equation}\label{E: edmd_iter}
    \uedmd (\cdot) := \Zk_N(\cdot) (\Tb^k\mathbf{w}), 
\end{equation}
\end{small}

\noindent where $\mathbf{w}$ is the $i$-th standard basis vector. Since  $\uedmd$ and $U$ are uniformly close, we can approximate the ROA (\ymr{w.r.t. the Hausdorff metric}) by the largest connected positive level sets of $\uedmd$. 

As for constructing a Lyapunov function, in virtue of \eqref{E: u_satisfaction} and Remark \ref{rem: lya-converge}, we may not directly use the result $\uedmd$ just yet. Fortunately, by the viscosity regularity of $U$ and the alternative comparison principles given in Lemma \ref{lem: facts}, one can seek a smooth function that uniformly approximates $\uedmd$, and hence uniformly approximates $U$, with its Lipschitz constant also converges to that of $U$. 

We achieve this extra modification by neural networks (NN) using a new set of samples $\{x^{(m)}\}$ (could be different from the one used in Algorithm \ref{alg: one_time}) and $\{\uedmd(x^{(m)})\}$. This NN approximation will be named  $\unn$. We omit this algorithm as it follows the standard procedure. The formal verification of $\unn$ utilizes satisfiability modulo theories (SMT) solvers and follows the exact procedures in \cite[Section V]{liu2023towards}.

\begin{rem}
Note that for any Lipschitz continuous function, the Lipschitz constant of its smooth neural approximations typically converges to the true Lipschitz constant. This phenomenon has been extensively studied by \cite{khromov2023some}. When verifying the Lipschitz constant of the residual is challenging, assuming this phenomenon holds is reasonable. \Qed
\end{rem}

We summarize the learning procedures in Fig.~\ref{fig: diagram}.
\begin{figure*}[!t]
\centerline{\includegraphics[scale = 0.50
]{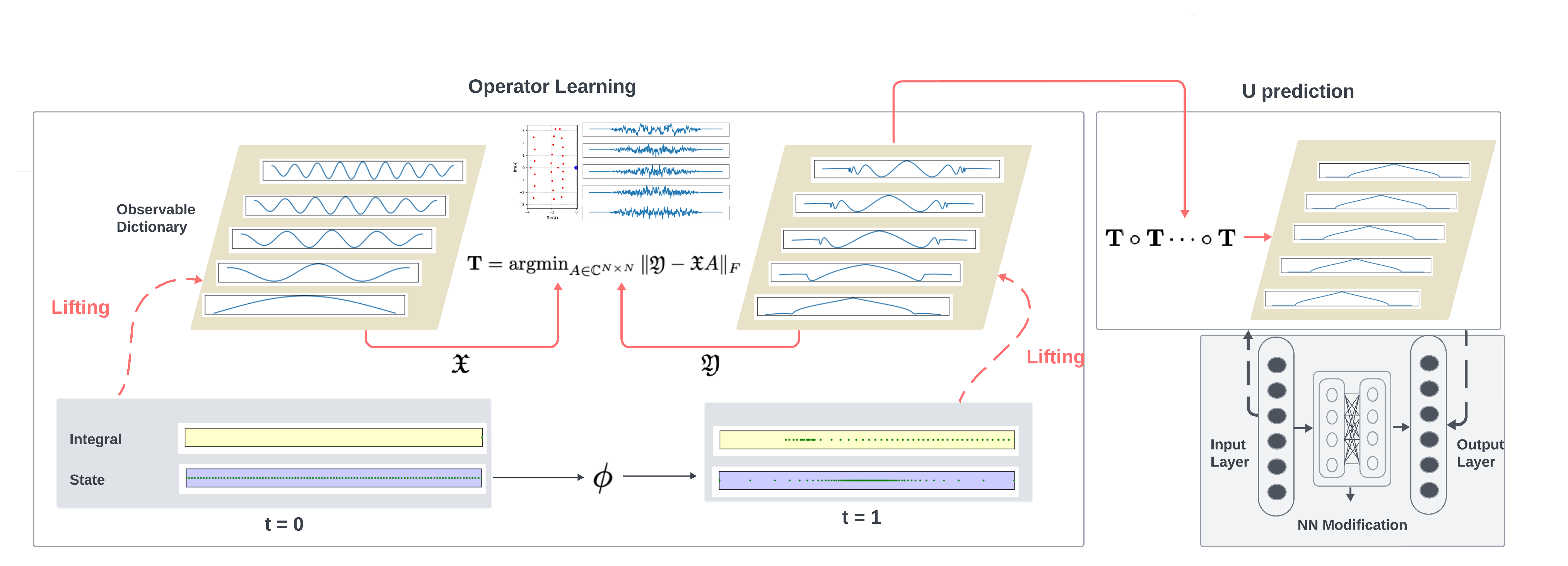}}
\caption{Zubov-Koopman approach of ROA prediction and Lyapunov function construction.}
\label{fig: diagram}
\end{figure*}

\section{Numerical Examples}
\label{sec: num}
In this section, we provide numerical examples to demonstrate
the proposed Zubov-Koopman approach for predicting ROAs, as well as learning and verifying neural
Lyapunov functions. For consistency, the reported computation
times were recorded on a MacBook Pro equipped
with an M1 Pro  10-Core CPU and 32GB
of memory. Multiprocessing was applied to all the examples, and the recorded calculation times for smaller data volumes can deviate from real values due to computational overhead. Code for the numerical examples can be found at \url{https://github.com/Yiming-Meng/Zubov-Koopman-Operator-Learning}. An illustration of the overall computational approach using a one-dimensional example can be found in   Appendix \ref{sec: num_1d}.

\subsection{Reversed Van der Pol Oscillator}
Consider the reversed Van der Pol oscillator
   \begin{equation*}
\dot \xb_1(t) = -\xb_2(t), \quad 
\dot \xb_2(t) = \xb_1(t) - (1 - \xb_1^2(t))\xb_2(t), 
\end{equation*} 
 with $\xb(0):=[\xb_1(0), \xb_2(0)]=[x_1, x_2]$. The system has a stable equilibrium point at $\zero$. We use this example to compare the data efficiency of predicting the ROA of $\set{\zero}$ as well as finding a Lyapunov function between the Zubov-Koopman approach and the data-driven approach presented in \cite{kang2021data}. 

In this example, we set  $\eta(x)=(x_1^2+x_2^2)/5$, and pick the  region of interest as $\rr=[-3, 3]^2$. A total of $M=200^2$, $250^2$, and $300^2$ uniformly
spaced samples $\set{x^{(m)}}$ in $\rr$ are generated respectively for three parallel experiments. 

\subsubsection{Zubov-Koopman method}
We simulate the trajectory up to $\Delta t= 1.5$. The dictionary $\Zk:= [\zk_{0, 0}, \zk_{0, 1}, \zk_{1,0}, \cdots,\zk_{i, j}, \cdots,\zk_{2N-1, 2N-1}]$ is selected in a similar manner as in the one-dimensional example  with adjustments for two-dimensional inputs, i.e., 
we set $\zk_{i, j}(x_1, x_2) = \cos\left(\frac{2\pi (i x_1+jx_2)}{3}\right )\exp\set{-(x_1^2+x_2^2)/4}$
for each $i$ and $j$. We let $N=50$, which means a total of $9801$ observable basis functions are used. To obtain $\uedmd$ through iterations, we set the termination tolerance to $10^{-2}$ with a maximum of $8$ iterations.  The prediction of the ROA (or $U$) follows the same procedure as in Section \ref{sec: num_1d}. 

For the construction and verification of a Lyapunov function, we need to incorporate an additional NN modification, as described in Section \ref{sec: lya_nn}, due to inaccuracies in the derivatives of $\uedmd$. The training set should be independent of the one used for obtaining $\uedmd$. Therefore, for all the experiments at this stage, we use $300^2$ uniformly spaced samples, along with the evaluations of $\uedmd$ at those points, to train $\unn$.  We use a network with $2$ hidden layers, each containing 15 neurons. We terminate
training when the mean-square training loss was smaller
than $10^{-8}$ or after $300$ epochs. 

\subsubsection{Data-Driven method}
As for the data-driven method in \cite{kang2021data}, for each experiment, we use the same samples as above and generate trajectory data up to a termination time $t = 10$, or until $\int_0^t \eta(\phi(s, x)) ds\geq 200$. We then use $\set{x^{(m)}}$ and $\left\{\exp\{-\int_0^t\eta(\phi(s, x^{(m)})) ds\}\right\}$ to train a neural-network $U_{\text{DD}}$. We use a network model with $2$ hidden layers, each containing 15 neurons. We terminate
training when the mean-square training loss is smaller
than $10^{-8}$ or after $500$ epochs. We predict the ROA as the largest connected positive level sets of $U_{\text{DD}}$. This neural solution $U_{\text{DD}}$ is also ready to be verified as a Lyapunov function without extra modification. 

\subsubsection{Verification}
Please note that the approximations $\unn$ and $U_{\text{DD}}$ are intended for Zubov's dual equation, not the original equation. To ensure consistency with the procedure for verifying the neural Lyapunov functions as described in \cite{liu2023towards}, it is important to pay attention to sign conversion.

\subsubsection{Experiment results}
The comparison results are reported in Table \ref{tab: nn_data}. In the table, `IVP solving' corresponds to steps 1) to 4) of Algorithm \ref{alg: data_generate}, while `stacking' involves preparing the data for learning. For the Data-Driven method, `function learning' refers to the NN training, whereas for the Zubov-Koopman approach, it pertains to the iterative procedure to obtain $\uedmd$ (see \eqref{E: edmd_iter}). We also recorded the NN final losses for both approaches whenever a neural network was employed during the process. 

We show the predicted and verified boundaries of the ROA for $M=100^2$ in Fig. \ref{fig: vdp}. The prediction using the Zubov-Koopman approach shows overall better accuracy, while the verified regions are similar for all the methods. On the other hand, with $M=100^2$ initial samples, the region verified by the Data-driven method is slightly smaller. This is because the verification relies on NN approximation for both methods, but the Zubov-Koopman approach's NN stage is based on an already good approximation of $U$, which is not constrained by the initial sample size.  These phenomena also reflect that the  $\uedmd$ and the data generated for the Data-driven method have similar quality in cases where $M=200^2$ and $300^2$.

\begin{figure}[!t]
\centerline{\includegraphics[scale = 0.45
]{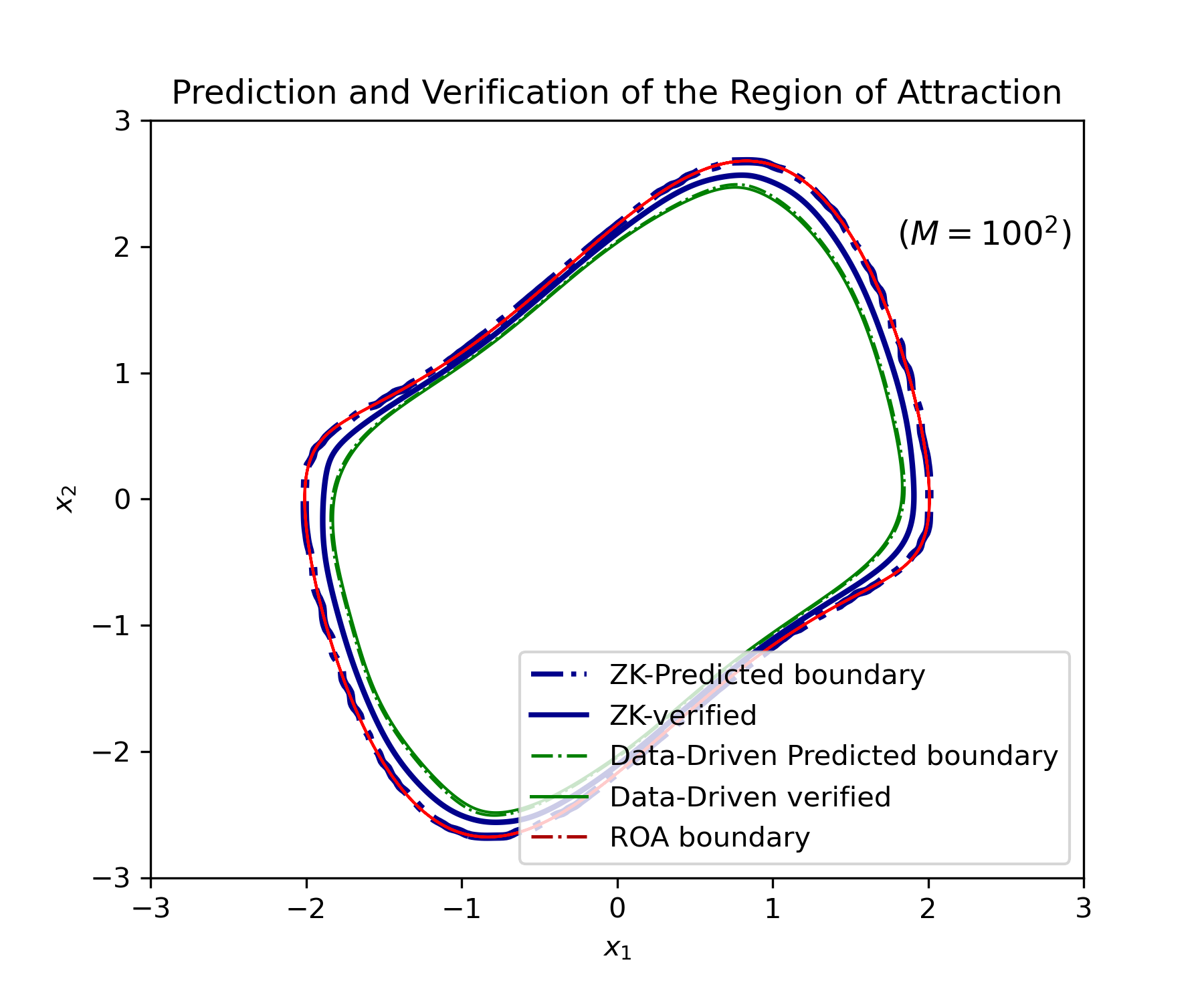}}
\caption{Predictions and verification of the ROA for the reversed Van der Pol with Zubov-Koopman and Data-driven method.}
\label{fig: vdp}
\end{figure}

\begin{table*}[h]
\caption{Verification of Neural Lyapunov Functions for Reversed Van der Pol Oscillator}
\centering
\begin{tabular}{|c|c|c|c|c|c|c|c|c|}
\hline
\textbf{Approaches} & \textbf{Sample size} & \makecell{\textbf{IVP solving}\\ \textbf{time}} & \textbf{Stacking time} & \makecell{\textbf{Operator}\\ \textbf{learning time}} & \makecell{\textbf{Function} $U$\\ \textbf{learning time}} & \makecell{\textbf{Modification}\\ \textbf{time ($300^2$ samples)}} & \makecell{\textbf{NN}\\ \textbf{final loss}}  & \makecell{\textbf{Verified}\\ \textbf{volume}}\\
\hline
Data-driven & $100\times 100$ & 10.13(s) & 0.81 (s) &  - & 54.24(s) & - & $8.48\times10^{-5}$ & 84.71\%\\
\hline
ZK + NN & $100\times 100$ & 7.68(s) & 5.40 (s) & 89.26(s) & 52.91(s) & 482.10 (s) & $1.25\times10^{-6}$ &  91.68\%  \\
\hline
\hline
Data-driven & $200\times 200$ & 40.83(s) & 0.20(s) &  - & 366.90(s) & - & $1.35\times10^{-6}$ & 88.87\% \\
\hline
ZK + NN & $200\times 200$ & 31.24(s) & 12.59(s) & 128.57 (s) & 53.18(s) & 482.10 (s) & $1.25\times10^{-6}$ & 90.32\% \\
\hline
\hline
Data-driven & $300\times 300$ & 87.06(s) & 0.32(s) & - & 545.02(s) & - & $9.37\times 10^{-7}$ & 87.09\% \\
\hline
ZK + NN & $300\times 300$ & 61.69(s) & 16.82(s) & 178.90(s) & 51.10(s)  & 486.67(s) & $1.05\times 10^{-5}$ & 89.70\% \\
\hline
\end{tabular}
\label{tab: nn_data}
\end{table*}

\subsection{Polynomial System}

We consider the polynomial system in \cite{mauroy2016global}:
\begin{equation*}
\dot \xb_1(t) = \xb_2(t), \quad 
\dot \xb_2(t) = -2\xb_1(t) + \frac13\xb_1^3(t) - \xb_2(t).
\end{equation*}

The origin of this system is known to have an unbounded domain of attraction, delimited by the stable manifolds of the saddle equilibrium points at $(\pm\sqrt{6},0)$ \cite{liu2023towards}. We restricted our computations to the domain $\rr=[-6,6]^2$, and used the proposed Zubov-Koopman approach to predict the refined ROA within $\inte(\rr)$. We also use NN modification to construct a Lyapunov function\footnote{In this case, this Lyapunov function is also a Lyapunov-barrier function. Previous work \cite{meng2022smooth, meng2023lyapunov} has shown that under mild conditions, Lyapunov and barrier functions can be united and behave as a single Lyapunov function.} to characterize both stability and safety. 

We use three experiments to investigate the effect of sample size and the width of the neural network:
(1) $M=300^2$, $\text{width}=15$;  (2) $M=500^2$, $\text{width}=15$; (3) $M=500^2$, $\text{width}=30$. In all of the experiments, we choose $\Delta t = 2$, $N=50$, and $\zk_{i, j}(x_1, x_2) = \cos\left(\frac{2\pi (i x_1+jx_2)}{6}\right )\exp\set{-(x_1^2+x_2^2)/25}$
for all $i,j\in\{0, 1, \cdots, 2N-1\}$. We use $2$-layer NN model for the modification stage. 
The verified portion of the refined ROA are $96.41\%$, $98.12\%$ and $98.50\%$. Fig. \ref{fig: poly1}  depicts the predicted and verified ROA with safety guarantees for the experiment (3). It can be seen that even though the actual ROA is unbounded, the proposed algorithm approximates the largest ROA that can be certified by the choice of neural Lyapunov function.

\begin{figure}[!t]
\centerline{\includegraphics[scale = 0.45
]{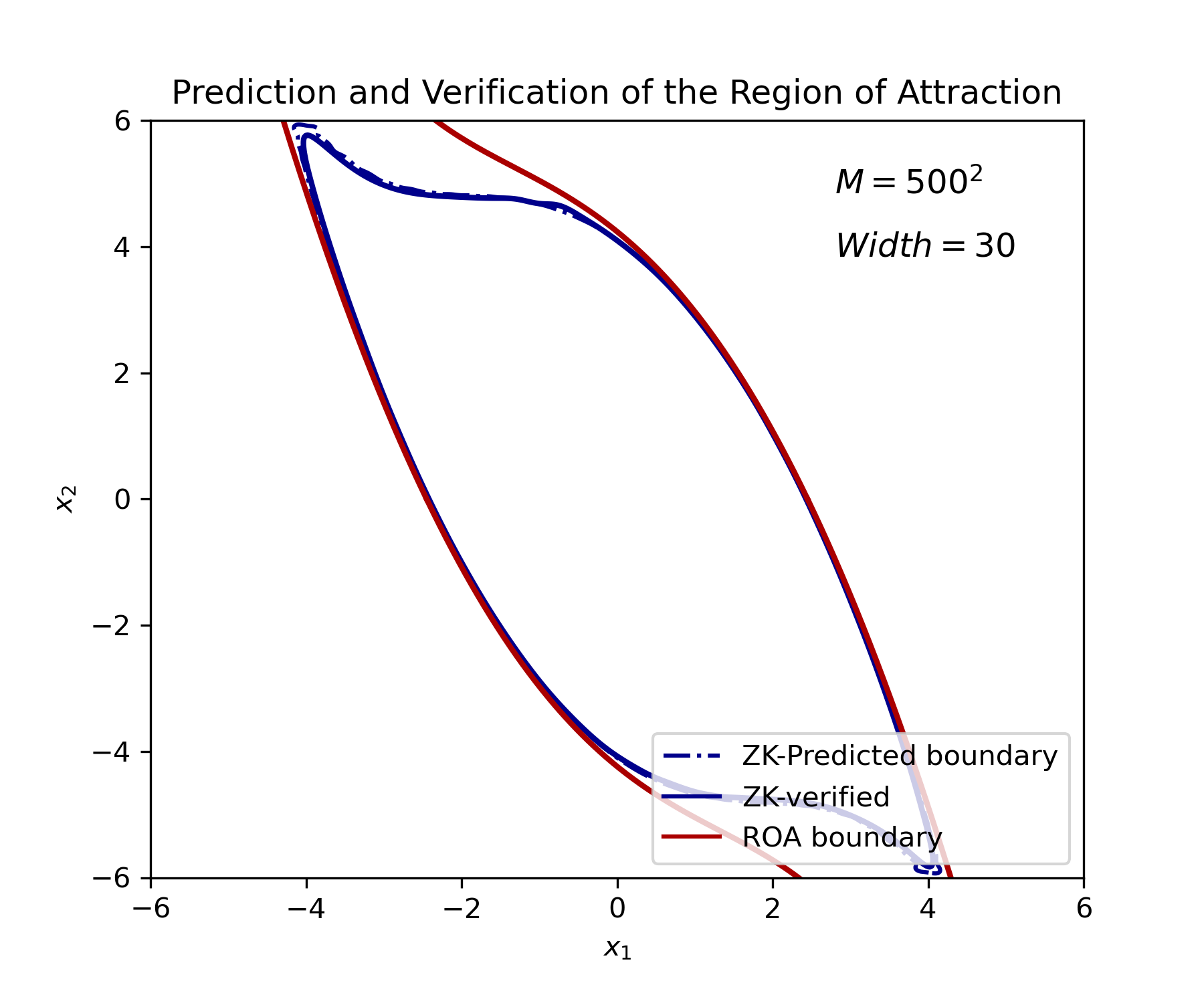}}
\caption{Predictions and verification of the  ROA with safety for the Polynomial System with Zubov-Koopman approach along with a 2-layer, 30-width neural network.}
\label{fig: poly1}
\end{figure}

\ym{\subsection{Two-Machine Power System}

Consider the two-machine power system \cite{vannelli1985maximal} modelled by $$\dot{\xb}_1(t) = \xb_2(t), \;\dot{\xb}_2(t) =  -0.5\xb_2(t) - (\sin(\xb_1(t) +\delta)-\sin(\delta)), $$
where $\delta = \frac{\pi}{3}$. We restrict our computations to the domain  $\mathcal{R}=[-2, 3]\times[-3,2]$ and use the proposed Zubov-Koopman approach to predict the refined ROA of the origin.  Note that the system has an unstable equilibrium point at $(\pi/3,0)$.   

In the experiment, we choose $\Delta t = 2$, $N=40$, $M=300^2$, and $\zk_{k, l}(x_1, x_2) = \exp\{i\frac{\pi}{6}(kx_1 + lx_2)\}$
for all $k,l\in\{-(N-1),  -(N-2),  \cdots, N-1\}$, where $i$ in this example is the imaginary unit, i.e., $i^2=-1$. We use a neural network with two hidden layers, 30 neurons each, for the modification stage.

The example (see Fig. \ref{fig: power}) shows that the proposed method outperforms sum-of-squares (SOS) Lyapunov functions~\cite{topcu2010help} when dealing with non-polynomial nonlinearities. 
Note that the verified ROA using dReal \cite{gao2013dreal} may appear smaller because the verification using an SMT solver tends to be more conservative, i.e.,  it only returns the largest level set within which the flow is forward invariant. However, to guarantee a subset of the Region of Attraction (ROA), forward invariance is not necessary. Readers can, based on their needs, use our approach as a data-driven numerical algorithm to predict (larger) ROA and roughly gauge where the boundary approximately lies.

\begin{figure}[!t]
\centerline{\includegraphics[scale = 0.46
]{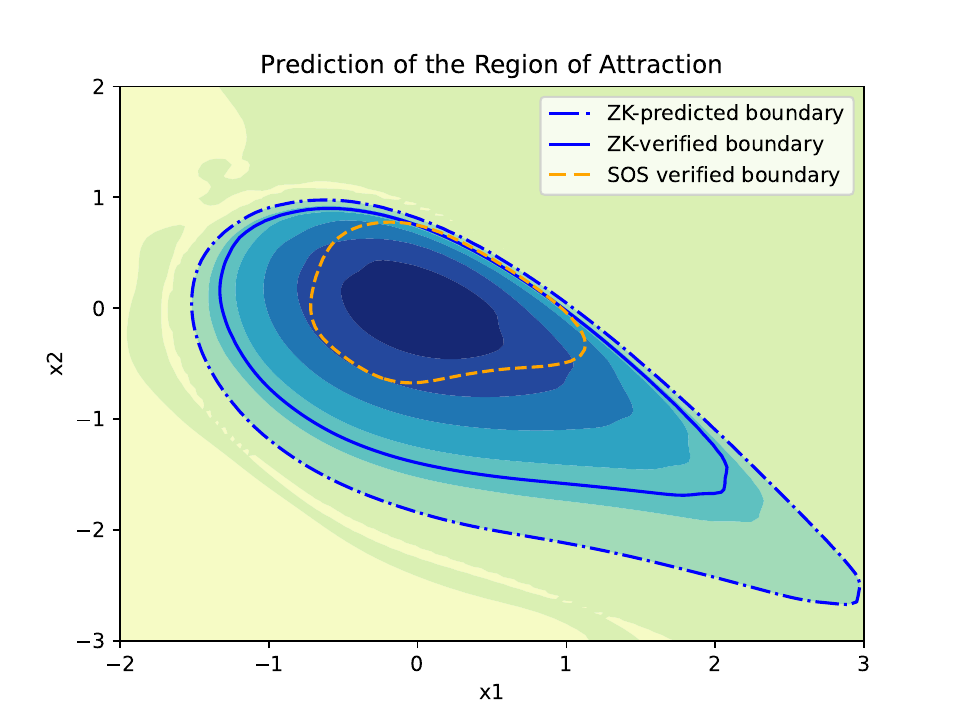}}
\caption{Verification of the  ROA with Safety for the Two-Machine Power System with Zubov-Koopman approach using $M=300^2$ samples.}
\label{fig: power}
\end{figure}

\subsection{A 3-Dimensional System}
We use this example to showcase the ability of the proposed method in predicting the ROA of a 3-dimensional system. Consider 
$$\left\{\begin{array}{lr} 
\dot \xb_1(t) = \xb_2(t) + 2\xb_2(t)\xb_3(t), \quad 
\dot \xb_2(t) = \xb_3(t),\\
\dot \xb_3(t) = -0.5\xb_1(t)-2\xb_2(t)-\xb_3(t). 
\end{array}\right.$$ 
We restrict our computations to the domain  $\mathcal{R}=[-2, 2]^3$ and use the proposed Zubov-Koopman approach to predict the refined ROA of the origin. In the experiment, we choose $\Delta t = 2$, $N=10$, $M=40^3$, and $\zk_{i,j,k}(x_1, x_2, x_3) = \exp\{i\frac{\pi}{6}(ix_1 + jx_2 + kx_3)\}$
for all $i,j,k\in\{-(N-1),  -(N-2),  \cdots, N-1\}$. We use a neural network with one hidden layer and $100$ neurons for the modification stage. The verified ROA within the region of interest is $\set{x\in\mathcal{R}: U_{\text{NN}}\geq 0.2}$. We present the $x_1-x_2$ section by performing the projection with $x_3$ fixed at $0$. Fig. \ref{fig: 3d} demonstrates that the verified ROA within $\mathcal{R}$ is almost the largest forward invariant set within $\mathcal{R}$. Again, it is worth mentioning that we restrict the ROA to \(\mathcal{R}\), which is within the actual ROA for the system. Therefore, the resulting ROA estimates from the proposed Zubov-Koopman approach and the SOS approach are not comparable (i.e., one is not a subset of the other). 
However, we would like to emphasize that the proposed Zubov-Koopman approach is performed to estimate the ROA for unknown systems (using a much less dense set of samples compared to previous examples), while the SOS technique requires exact model information.}

\begin{figure}[!t]
\centerline{\includegraphics[scale = 0.46
]{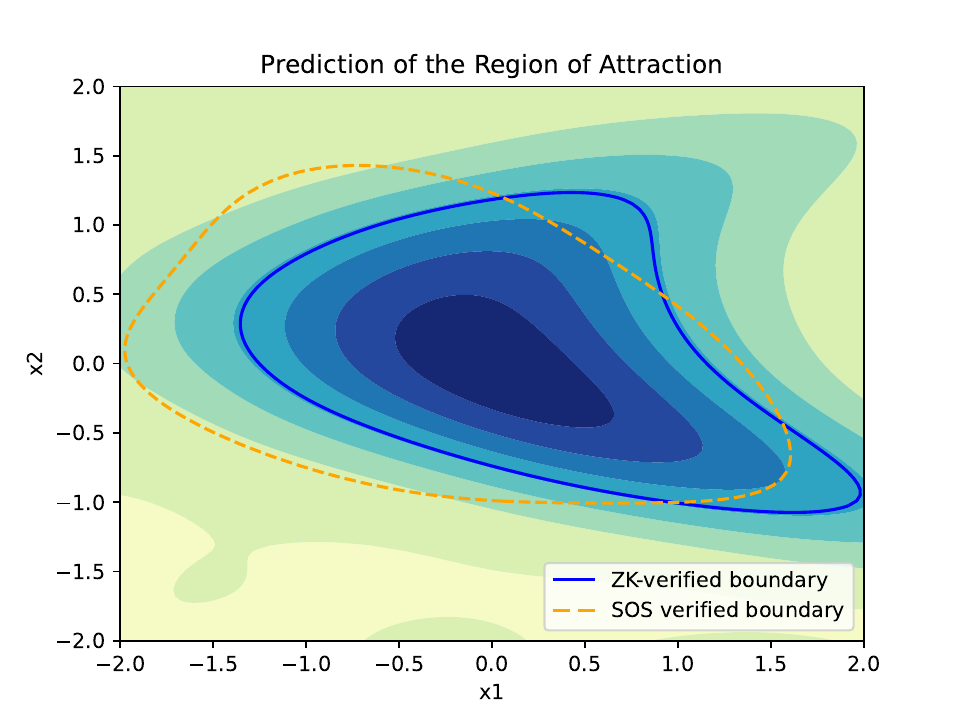}}
\caption{Verification of the  ROA with safety for the 3D System   with Zubov-Koopman approach using $M=40^3$ samples.}
\label{fig: 3d}
\end{figure}

\subsection{Stiff System I}
In this example, we examine Van der Pol oscillators
\begin{equation*}
\dot \xb_1(t) = -\xb_2(t), \quad 
\dot \xb_2(t) = \xb_1(t) - \mu(1 - \xb_1^2(t))\xb_2(t), 
\end{equation*}
with $\mu = 4$ and $6$. 
The value of $U$ changes rapidly due to the high evolution speed of the state within the stiff systems. For both systems, we use $M=300^2$ samples to determine $\uedmd$ and employ a network model with two hidden layers and 15 neurons for modification. However,  the rapid changes prevent the validation of a  Lyapunov function. We present the color map of $\uedmd$ for predicting the ROAs in Fig. \ref{fig: stiff1}. We can observe that as $\mu$ increases, the same number of samples becomes insufficiently dense to provide adequate information for operator learning. This results in substantial oscillations at the boundaries of sub-level regions.

\begin{figure}[!t]
\centerline{\includegraphics[scale = 0.3
]{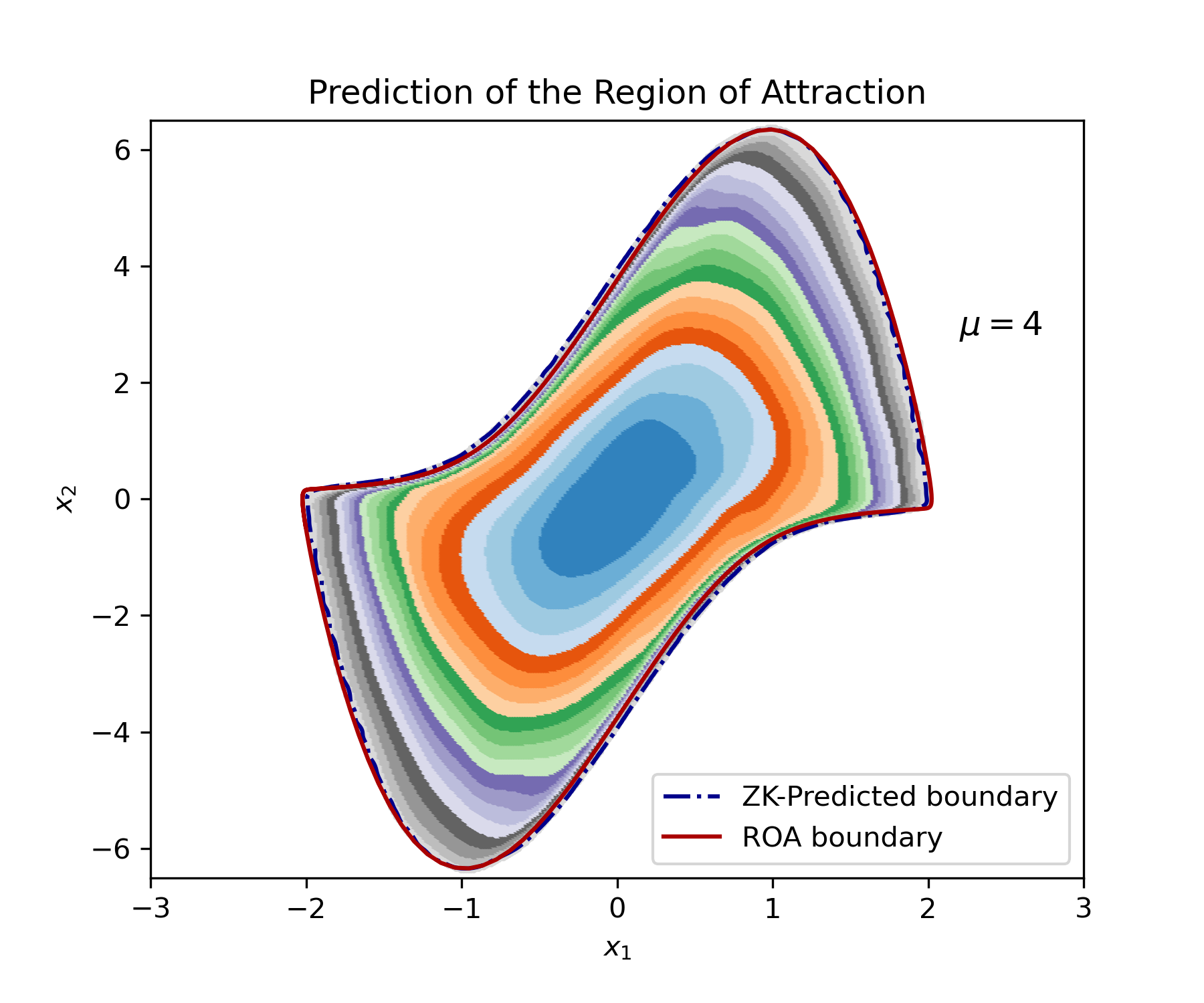}\includegraphics[scale = 0.3
]{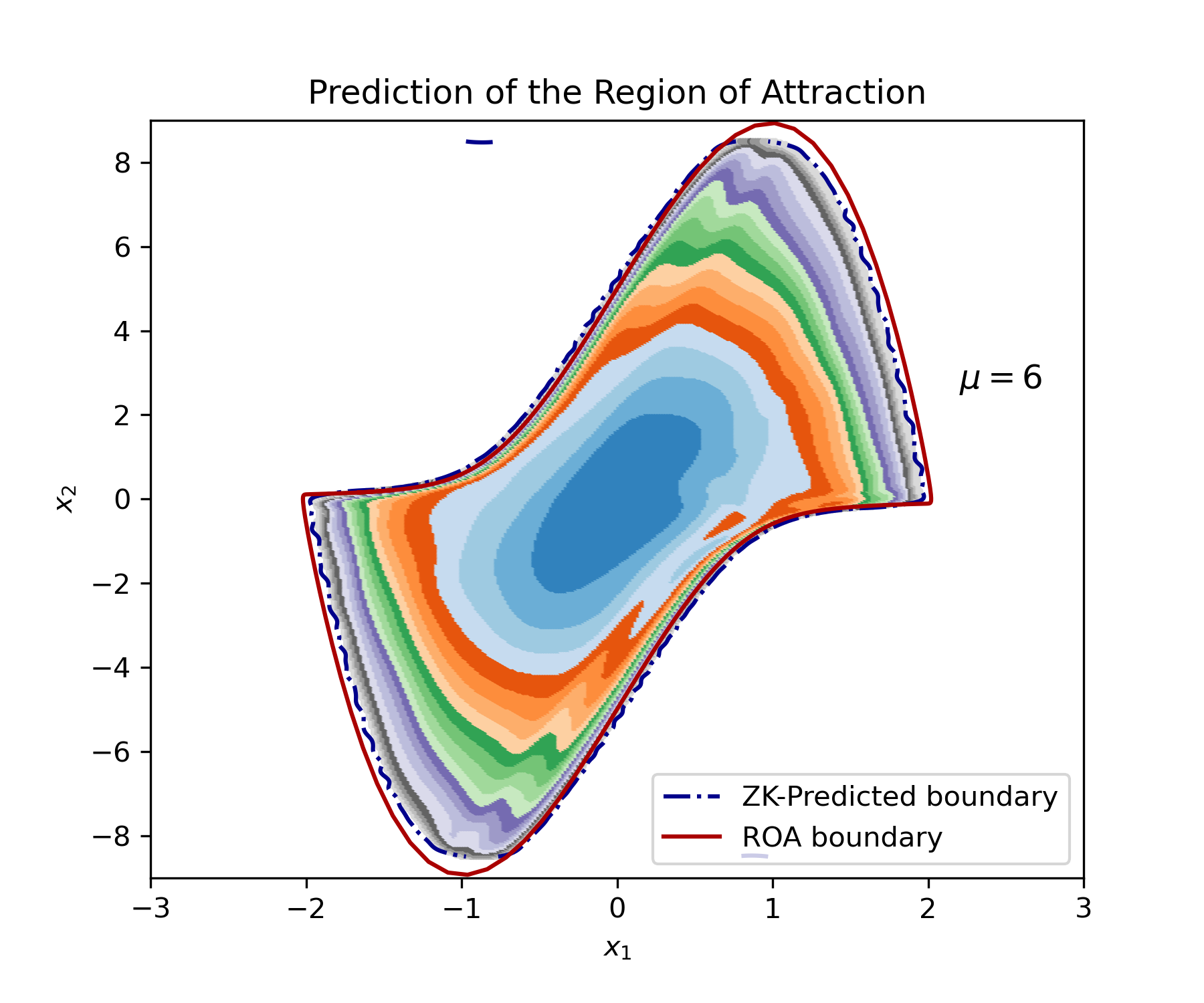}}
\caption{Zubov-Koopman approximations of the ROAs for stiff Van der Pol systems. }
\label{fig: stiff1}
\end{figure}

\subsection{Stiff System II}
We use this example to compare the predictability of ROA  from \cite[Example 3]{mauroy2013spectral} for dynamics
\begin{equation*}
\dot \xb_1(t) = -2\xb_1(t)+\xb_1^2-\xb_2^2, \quad 
\dot \xb_2(t) = -2.5\xb_2(t) + 2\xb_1(t)\xb_2(t).
\end{equation*}
The origin is globally stable on $\R^2\setminus\{x\in\R^2: x_1\geq 2, x_2=0\}$. We choose $\rr=[-4, 4]^2$  
and use $M=300^2$ samples to learn $\uedmd$. The non-zero color map represents the ROA relative to $\rr$, as shown in Fig. \ref{fig: poly3}. The learned $\uedmd$ demonstrates better predictability than the local Lyapunov functions generated by conventional Koopman operators. We omit the NN-modification and formal verification in this example to make a fair comparison with \cite{mauroy2013spectral}. 

\begin{figure}[!t]
\centerline{\includegraphics[scale = 0.48
]{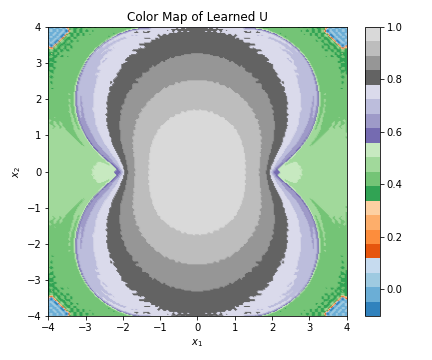}}
\caption{Predictions and verification of the  ROA with safety for the Stiff System II with Zubov-Koopman approach using $M=300^2$ samples.}
\label{fig: poly3}
\end{figure}

\section{Conclusion}\label{sec: conclu}
In this paper, we introduced a Zubov-Koopman approach to characterize the solution to Zubov’s equation,  based on regularity and convergence analysis for a broader range of nonlinear systems. The proposed operator exhibits a semigroup property, and the learning technique is similar to the conventional Koopman operator. In particular, we executed an EDMD-like algorithm using a Fourier-like dictionary of observable functions based on our understanding of the solution properties.  We then devise an iterative approach that employs the learned Zubov-Koopman operator to approximate the solution to Zubov’s equation, achieving a high level of accuracy. Compared to the existing data-driven methods,  our proposed technique demonstrates a better predictability of the ROAs via numerical examples. 

It is worth mentioning that, through numerical examples, we observed clear spectral gaps in the learned operators. \ymr{The effectiveness of using semigroup properties with such spectral gaps to find low-dimensional approximations for characterizing the long-term fixed points of PDEs has been demonstrated in the literature \cite{blomker2004multiscale,  meng2023hopf}. Similarly, this approach} 
can potentially lead to dimension reduction for the Zubov-Koopman operators and effectively address the `curse of dimensionality' for unknown systems. Future directions will be on the spectral analysis with applications to high-dimensional physical systems.

One potential limitation of this work is its predictability of the flow in the state space. We have to rely on the knowledge of fixed points to predict the ROAs. \ymr{However, recent work  \cite{meng2024koopman} has provided a way to use the same set of training data as this paper for identifying system transitions, achieving improved accuracy compared to the benchmark Koopman-logarithm based approach \cite{mauroy2019koopman}. It would be interesting to investigate how to align this approach with dimension reduction analysis and the proposed method in this paper, so that a streamlined computational tool can be developed for situations with limited information about stable attractors.} 
Another immediate area for future work with the Zubov-Koopman framework could involve systems with robust control or other types of measurable inputs. It would be intriguing to explore its potential in predicting controllable regions with stability and safety properties for unknown systems. 

\appendices

\section{Fundamental Properties of Viscosity Solutions}\label{sec: app1} 
We provide some fundamental properties of viscosity solutions in this appendix. 
The following lemma \cite[Lemma 1.7, Lemma 1.8, Chapter I]{bardi1997optimal} provides some insights on $\partial^+v(x)$ and $\partial^+v(x)$ for some $v\in C(\X)$. 

\begin{lem}[Sub- and Supperdifferential]\label{lem: facts}
Let $v\in C(\X)$. Then
\begin{itemize}
    \item[(1)] $p\in \partial^+v(x)$ if and only if there exists $\nu\in C^1(\X)$ such that $\nabla\nu(x)=p$ and $v-\nu$ has a local maximum at $x$;
    \item[(2)] $q\in \partial^-v(x)$ if and only if there exists $\nu\in C^1(\X)$ such that $\nabla\nu(x)=q$ and $v-\nu$ has a local minimum at $x$;
    \item[(3)] if for some $x$ both $\partial^+v(x)$ and  $\partial^-v(x)$ are nonempty, then $\partial^+v(x)=\partial^-v(x)=\{\nabla v(x)\}$;
    \item[(4)] the sets  $\{x\in\X: \partial^+v(x)\neq \emptyset\}$ and $\{x\in\X: \partial^-v(x)\neq \emptyset\}$ are dense. 
\end{itemize}
\end{lem}

In view of (1) and (2) in Lemma \ref{lem: facts}, (1) and (2) in Definition \ref{def: vis1} are equivalent as
\begin{itemize}
    \item[(1)] for any $\nu\in C^1$, if $x$ is a local maximum for $v-\nu$, then $F(x, \nu(x), \nabla\nu(x))\leq 0$;  
    \item[(2)] for any $\nu\in C^1$, if $x$ is a local minimum for $v-\nu$, then $F(x, \nu(x), \nabla\nu(x))\geq 0$. 
\end{itemize}

\section{Unique Bounded Viscosity Solutions 
}\label{sec: app_uniq} 

\textit{\bf Proof of Theorem \ref{thm: unique}:} The uniqueness property within $\roa(\A)$ follows the result \cite[Thoerem 3.8]{camilli2001generalization} for zero  perturbation of \eqref{E: sys}. Alternatively, we can suppose there exists another  viscosity solution $\tilde{V}$ to $-\nabla V\cdot f(x)-\eta(x)=0$ with $V(\xeq) = 0$. Let $w = V-\tilde{V}$. Then, $w$ solves $-\nabla w(x)\cdot f(x) = 0$ with $w(\xeq)=0$ in a viscosity sense. One can follows a similar proof as in \cite[Chapter II, Proposition 5.18]{bardi1997optimal} to show that $w(\phi(t, x))=w(x)$ for all $t\geq 0$ and all $x\in\roa(\A)$. By the stability assumption, one has $w\equiv w(\xeq) = 0$, which shows that $V$ is the unique viscosity solution on $\roa(\A)$. For any valid $h$, by the definition \eqref{E: U} 
and by  \cite[Chapter
2, Proposition 2.5]{bardi1997optimal}, we immediately have that $U_h$ is the unique viscosity solution to \eqref{E: dual_test} on $\roa(\A)$.

Now we suppose  $\roa(\A)\neq \R^n$. Then, $\partial\roa(\A)\neq \emptyset$.  By the above uniqueness argument within $\roa(\A)$,  and  by the construction of $U_h$, all viscosity solutions to \eqref{E: dual_test} should be equal to $0$ on $\partial\roa(\A)$.  We then work on $\R^n\setminus\overline{\roa(\A)}$ and verify that $0$ is the unique bounded viscosity solution to \eqref{E: dual_test}. 

Let $\Omega=\R^n\setminus\overline{\roa(\A)}$, then, $\partial\Omega = \partial \roa(\A)$.  
Suppose $u$ is another bounded viscosity solution to \eqref{E: dual_test} and let $\Psi(x,y)=u(x)- |x-y|^2/(2\eps)$ for any $\eps>0$.  Then, $\lim_{|x|+|y|\ra \infty}\Psi(x, y)=-\infty$. 
By the continuity of $\Psi$, there exists  $x_\eps, y_\eps\in\overline{\Omega}$ such that $ \Psi(x_\eps, y_\eps)=\max_{\overline{\Omega}\times\overline{\Omega}} \Psi(x,y)$.  
Then, for any $\eps>0$, $\max_{x\in\overline{\Omega}} u(x) = \max_{x\in \overline{\Omega}} \Psi(x,x)\leq 
    \Psi(x_\eps,y_\eps)\leq u(x_\eps)$. 
In addition, by the exact same argument as in  \cite[Chapter II, Theorem 3.1]{bardi1997optimal}, one has that $|x_\eps-y_\eps|\ra 0$ and $\frac{|x_\eps-y_\eps|^2}{2\eps}\ra 0 $ as $\eps\ra0$. 

Now we argue that  
$u(x_\eps)\leq 0$ as $\eps\ra 0$, which will imply that $u(x)\leq 0$ on $\Omega$.  If $x_\eps\in\partial\Omega$, then $u(x_\eps)=0$. 
If $y_\eps\in\partial\Omega$, then 
$u(x_\eps)= u(x_\eps) - u(y_\eps)\ra 0$ by the continuity of $u$. 
If $x_\eps, y_\eps\in\Omega$, set $\psi_1(y):=u(x_\eps)-\frac{|x_\eps-y|^2}{2\eps}$ and $\psi_2(x):=\frac{|x-y_\eps|^2}{2\eps}$. Then, $x_\eps$ is a local maximum for $u-\psi_2$ and $y_\eps$ is a local minimum for $U_h-\psi_1$. Note that $\nabla\psi_1(y_\eps)= \frac{x_\eps-y_\eps}{\eps}=\nabla\psi_2(x_\eps)$.  In addition, since $u$ and $U_h$ are viscosity (sub/supper) solutions on $\Omega$, we have that $\eta(x_\eps)u(x_\eps)-\frac{x_\eps-y_\eps}{\eps}\cdot f(x_\eps)\leq 0$ and $\frac{x_\eps-y_\eps}{\eps}\cdot f(y_\eps)\leq 0$. 
This implies $u(x_\eps)\leq \frac{x_\eps-y_\eps}{\eps\eta(x_\eps)}\cdot (f(x_\eps)-f(y_\eps))$ since $\eta(x_\eps)>0$. 
By the  continuity of $f$, we have $u(x_\eps)\leq 0$  as $\eps\ra 0$.

The other side, i.e., $u(x)\geq 0$ on $\Omega$, can be proved in the same manner. Therefore, $u(x)=0$ on $\Omega$. \pfbox


\section{Fundamental Properties of Zubov-Koopman}\label{sec: app2}
In the appendix, we complete the proof for Proposition \ref{prop: semigroup_eigenfunction}.
\begin{proof}
It is clear that $\T_0h(x)=h(x)$. We also have the following identities. 
\begin{small}
        \begin{equation*}
    \begin{split}
        &\T_s\circ\T_th(x) \\
        =  &\exp\set{-v_s(x)}\T_th(\phi (s, x))\\
         = &\exp\set{-v_s(x)} \exp\set{-\int_0^t\eta(\phi (r, \phi (s, x)))dr}h(\phi (t, \phi (s, x)))\\
         = &\exp\set{-v_s(x)}\exp\set{-\int_0^t\eta(\phi (r+s, x)))dr}h(\phi(t+s, x))\\
          = &\exp\set{-v_s(x)}\exp\set{-\int_s^{t+s}\eta(\phi(\sigma, x))d\sigma}h(\phi(t+s, x))\\
          = &\exp\set{-\int_0^{t+s}\eta(\phi(\sigma, x))dr}h(\phi(t+s, x)) =\T_{s+t}h(x),
    \end{split}
\end{equation*}
\end{small}
which show the associativity of $\{\T_t\}_{t\geq 0 }$. The strong continuity of $\{\T_t\}_{t\geq 0 }$ can be  verified straightforwardly by Definition \ref{def: semigroup}. For the second part of the proof, we apply dynamic programming and use the similar technique as above. We omit the details due to the similarity. 
\end{proof}

\section{Illustration of the Zubov-Koopman Approach with a One-Dimensional Nonlinear System}

\subsection*{One-dimensional Nonlinear System}\label{sec: num_1d}

We revisit Example \ref{eg: 1d} and explain the effectiveness of the Zubov-Koopman approach. Consider the dynamical system 
\begin{small}
    $$\dot{\xb}(t)=-\xb(t) + \xb^3(t), \;\;\xb(0)=x, $$
\end{small}
and  $\eta(x)=|x|$. We would like to find the ROA of $\{0\}$ and a valid Lyapunov function using the trajectory data up to $\Delta t = 1$ within the region of interest $\rr=[-1.5, 1.5]$. We sample $\{x^{(m)}\}_{m=0}^{M-1}\subset\rr$ for $M=3001$, and use the  dictionary $\Zk=[\zk_0, \zk_{\pm 1}, \cdots, \zk_{\pm (N-1)}]$ for $N=512$, where  $    \zk_i = \cos\left(\frac{2\pi i x}{3}\right )\exp\set{-x^2/4}, i\in\set{0, \pm1, \cdots, \pm (N-1)}.$ 
We execute Algorithm \ref{alg: 
one_time} to prepare the training set, which takes $15.086$ seconds. We obtain the discrete operator $\Tb$ as in Section \ref{sec: edmd}, which requires $0.309$ seconds. The spectrum ($\log$-scale) and the first $5$ learned functions (of $\T_t^\eps$ as defined in Section \ref{sec: compact_operator}) are shown in Fig.~\ref{fig: eig}. We can clearly see the spectral gap between the top eigenvalue (in blue) and the rest of the point spectrum.  

\begin{figure}[!t]
\centerline{\includegraphics[scale = 0.48
]{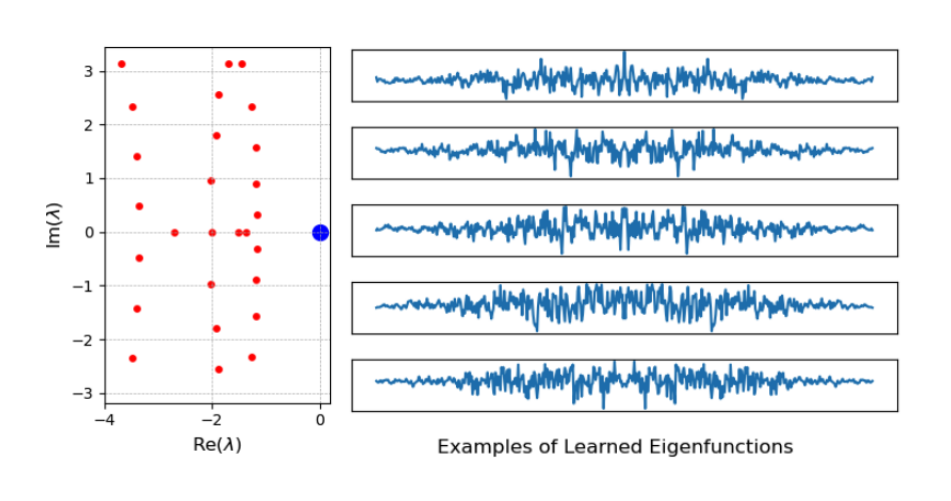}}
\caption{Selected eigenvalues (log-scale) and eigenfunctions of the learned operator.}
\label{fig: eig}
\end{figure}
\begin{rem}
For each $i$,  $\zk_i(0)=1$. 
As discussed in Section \ref{sec: time-series}, we have  $\lim_{t\ra\infty}\T_t\zk_i= U$. We simply pick $\zk_0$ as for our purpose.  Furthermore, the choice of $\set{\cos\left(\frac{2\pi i x}{3}\right )}$ is in consideration of the periodicity of the (stopped) trajectories in $\rr$. One can also replace it with Fourier basis. The multiplication with $\exp\set{-x^2/4}$ is for smoothing the high-frequency observables functions. \Qed
\end{rem}

Using the learned matrix $\Tb$, we first show the evolution of $\zk_0$ governed by the approximation of $\T_\Delta$, i.e., we obtain $\Zk_{2N-1}(\cdot) \Tb^k\mathbf{w}$, where $\mathbf{w}=[1, 0, \cdots, 0]^T$, for a sequence of $k$.  We observe the trend of convergence as in Fig.\ref{fig1}.   
\begin{figure}[!t]
\centerline{\includegraphics[scale = 0.43 
]{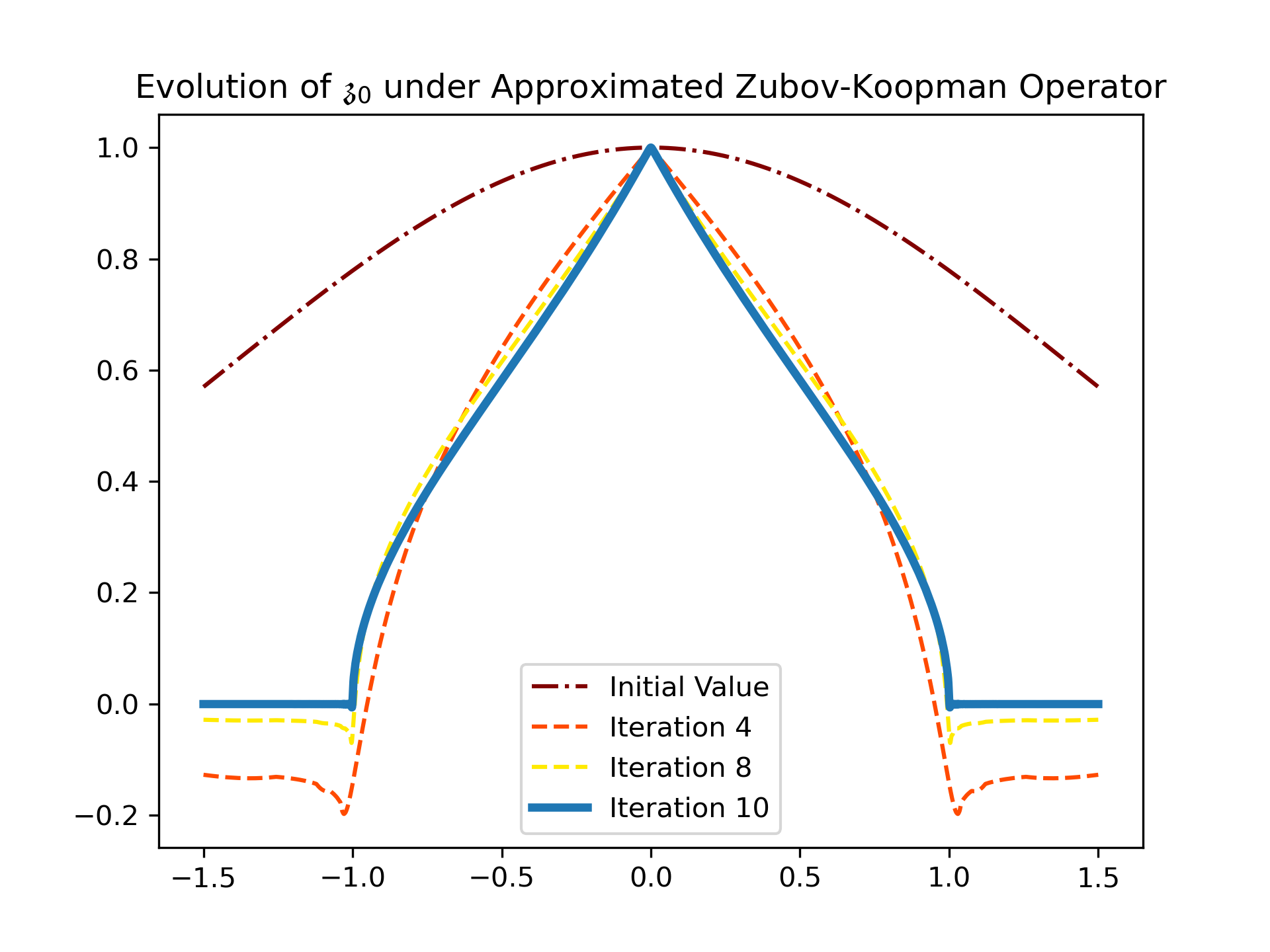}}
\caption{Evolution and convergence of $\zk_0$ under the $k$-th iteration of Operator $\Tb$.}
\label{fig1}
\end{figure}

After the $10$-th iteration, we have that $\|\Tb^{10}-\Tb^{9}\|_F\approx 0.05$. Now we let $\uedmd(\cdot) =\Zk_{2N-1}(\cdot) \Tb^{10}\mathbf{w}$ and obtain  that the largest connected positive level sets of $\uedmd$ is $\set{x\in\rr: \uedmd(x)\geq 0.0012}$, which is approximately $[-0.99975, 0.99975]$. On this domain, we set $V_{\operatorname{ZK}}=-\log(\uedmd)$. 
Then, we 
compare the $\uedmd$ and $V_{\operatorname{ZK}}$ with the true $U$ and $V$ defined in Example \ref{eg: 1d} as in Fig.~\ref{fig: U_1d}. It can be inspected that both approximations achieve a high accuracy. The under-approximation of the ROA is due to the tiny oscillation near the boundary points $\set{\pm1}$.
\begin{figure}[!t]
\centerline{\includegraphics[scale = 0.3
]{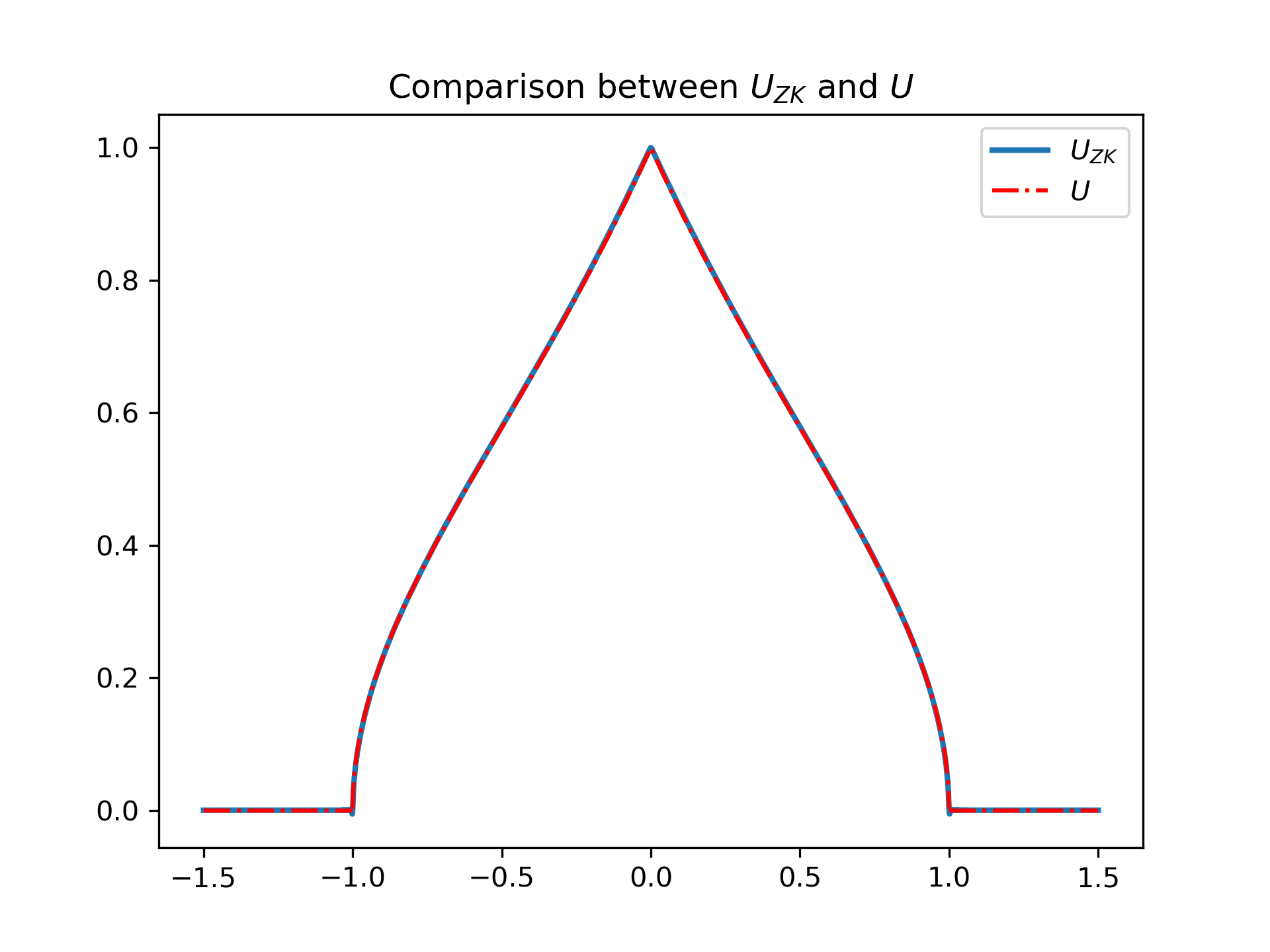}\includegraphics[scale = 0.3
]{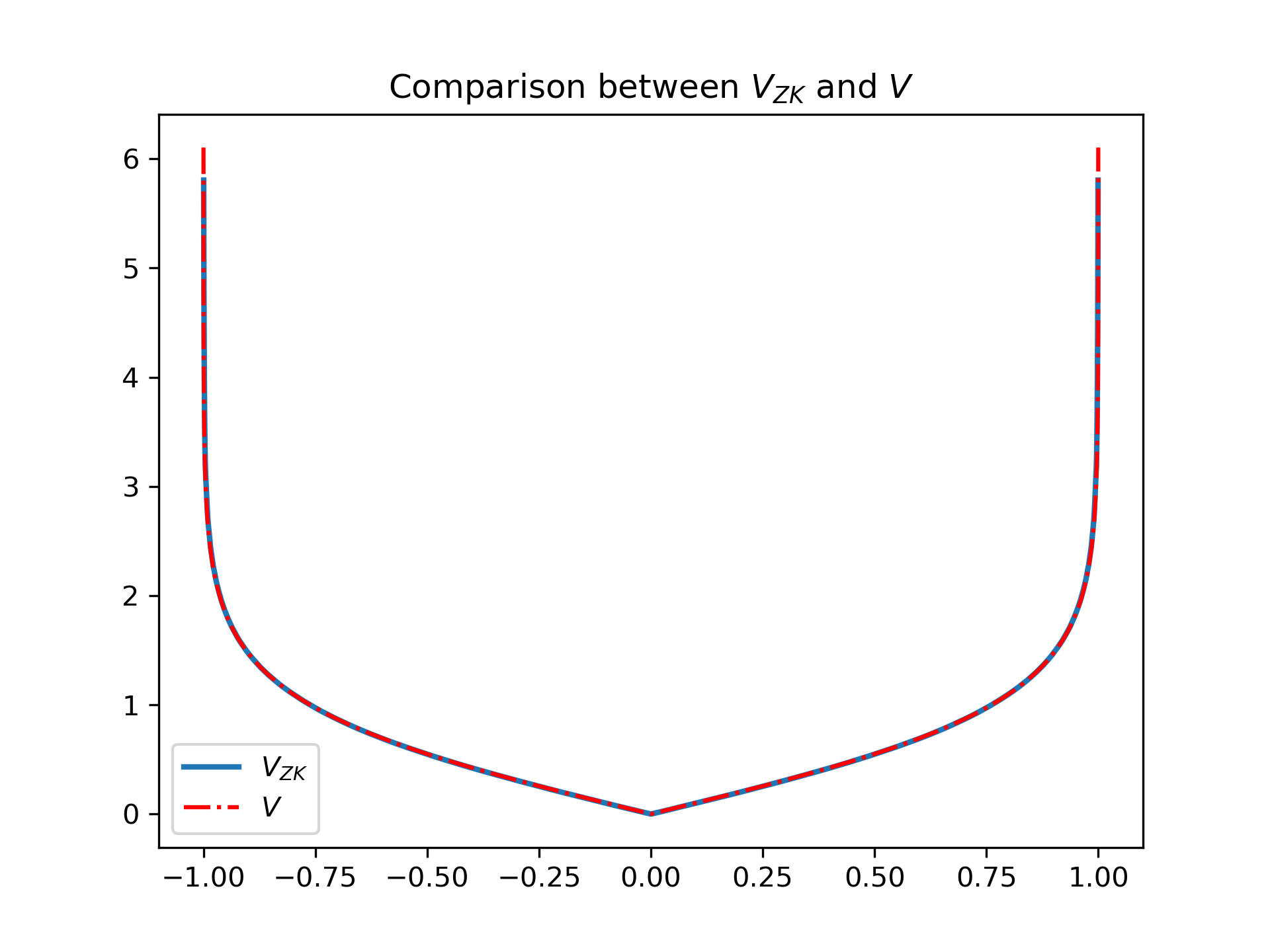}}
\caption{Comparisons between the true functions and Zubov-Koopman approximations. }
\label{fig: U_1d}
\end{figure}

\begin{rem}
    It is worth noting that, as shown in Figure \ref{fig: eig}, the first (complex-valued) eigenfunction corresponding to the eigenvalue $e^{0t}=1$  is nowhere close to the real-valued true eigenfunction $U$ of $\T_\Delta$. This phenomenon aligns with our explanation in Remark \ref{rem: extension}. \Qed
\end{rem}

To see whether $V_{\operatorname{ZK}}$ can serve as a Lyapunov function, we proceed to verify if the numerical error can still make $\L_f V_{\operatorname{ZK}}<0$ (or equivalently, $\L_f \uedmd>0$). The comparisons with the true functions are provided in Fig. \ref{fig: Lf}. 
\begin{figure}[!t]
\centerline{\includegraphics[scale = 0.3
]{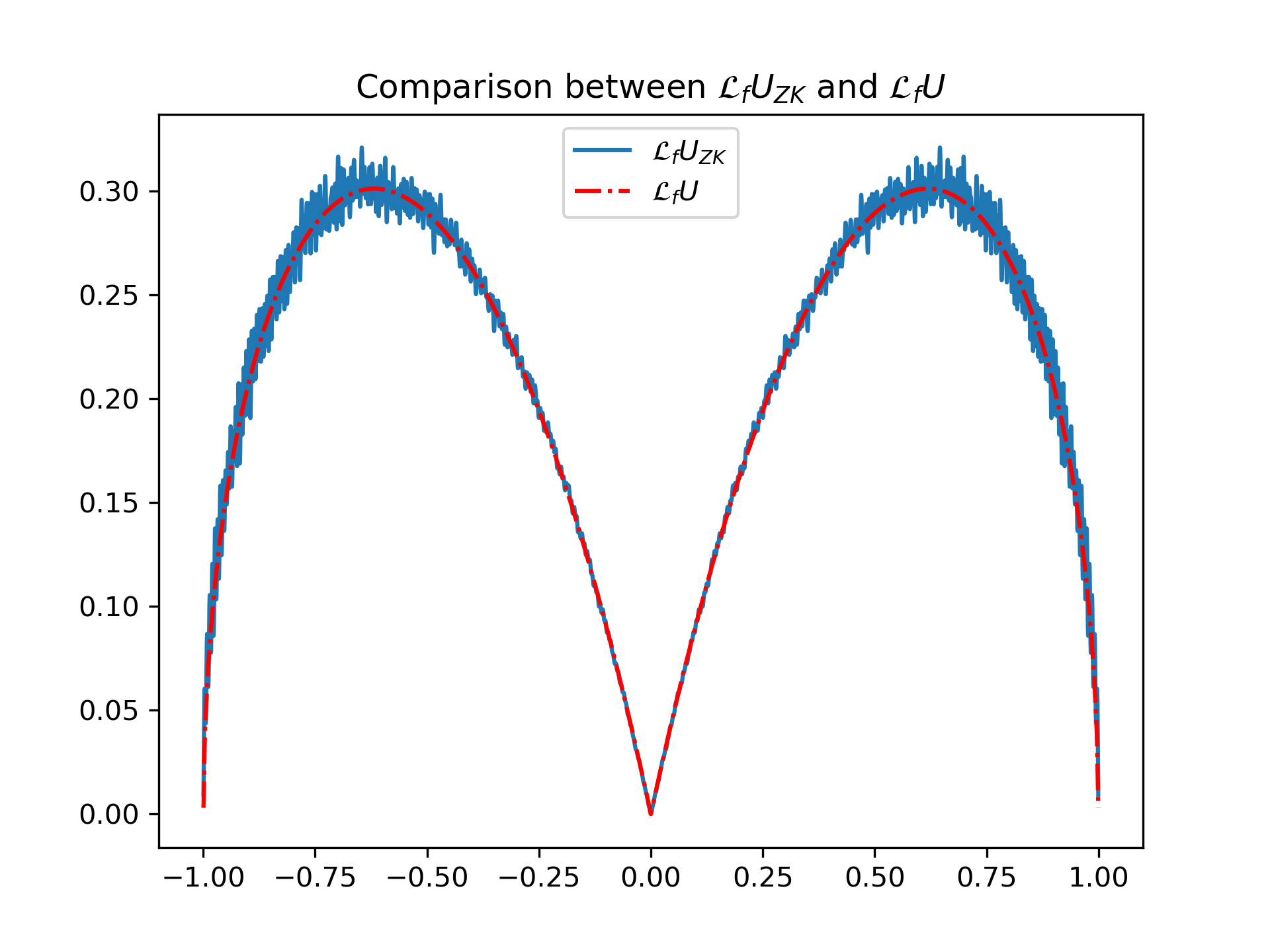}\includegraphics[scale = 0.3
]{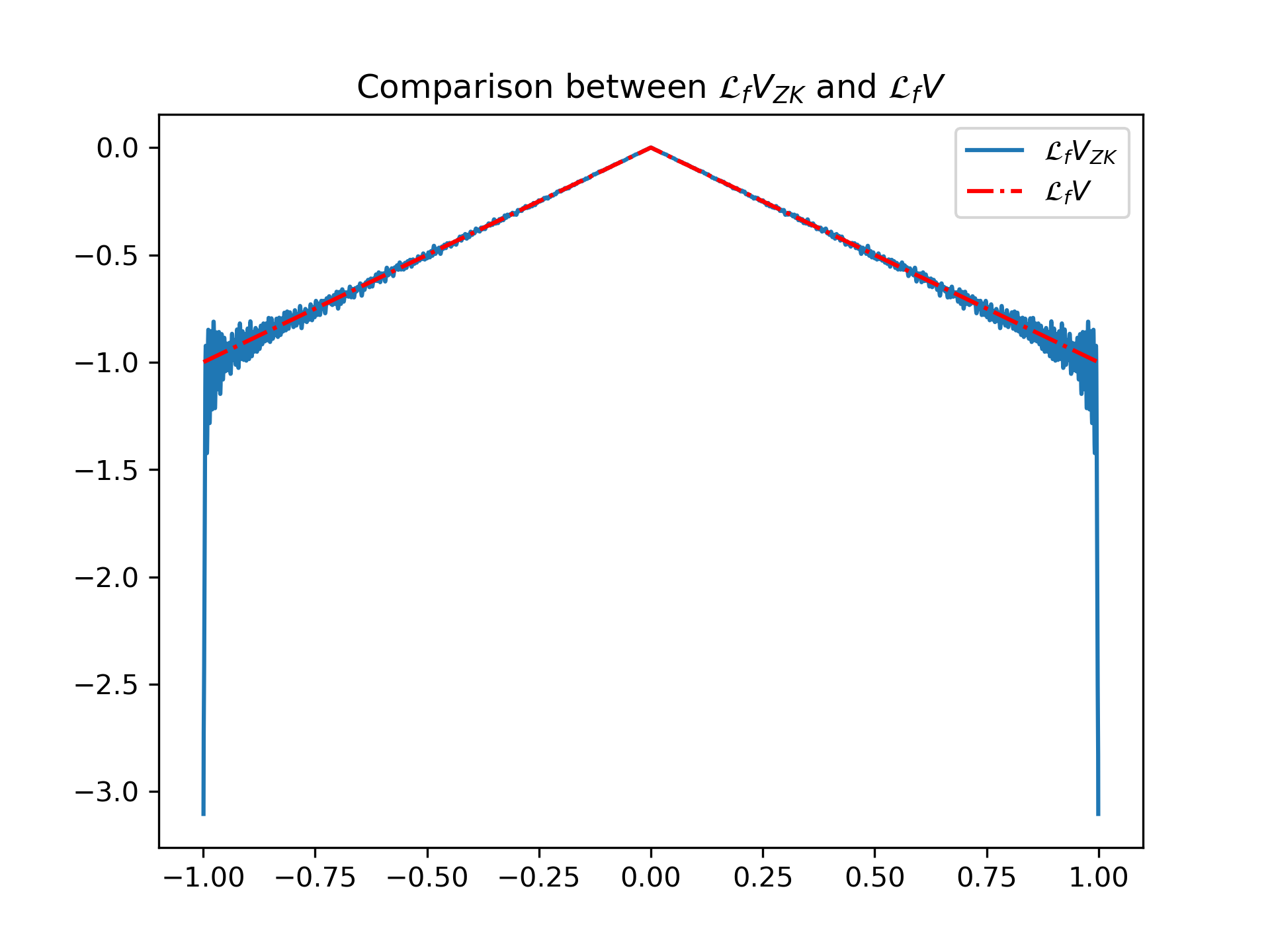}}
\caption{Comparisons between the Lie derivatives and Zubov-Koopman approximations. }
\label{fig: Lf}
\end{figure}
The approximations for $\L_f U$ and $\L_f V$ are not as accurate as those for $U$ and $V$ due to the high-frequency decomposition. However, the approximate quantities are verified to preserve the sign on  $[-0.9995, 0.9995]$, which is the valid sub-domain to use $V_{\operatorname{ZK}}$ as a Lyapunov function. We hence do not have to use an extra NN-modification for the construction. 


\bibliographystyle{ieeetr} 
\bibliography{ref}

\begin{thebibliography}{10}

\bibitem{li2020robustly}
Y.~Li and J.~Liu, ``Robustly complete synthesis of memoryless controllers for nonlinear systems with reach-and-stay specifications,'' {\em IEEE Transactions on Automatic Control}, vol.~66, no.~3, pp.~1199--1206, 2020.

\bibitem{wilson1969smoothing}
F.~Wilson, ``Smoothing derivatives of functions and applications,'' {\em Transactions of the American Mathematical Society}, vol.~139, pp.~413--428, 1969.

\bibitem{lin1996smooth}
Y.~Lin, E.~D. Sontag, and Y.~Wang, ``A smooth converse {L}yapunov theorem for robust stability,'' {\em SIAM Journal on Control and Optimization}, vol.~34, no.~1, pp.~124--160, 1996.

\bibitem{clarke1998asymptotic}
F.~H. Clarke, Y.~S. Ledyaev, and R.~J. Stern, ``Asymptotic stability and smooth {L}yapunov functions,'' {\em Journal of differential Equations}, vol.~149, no.~1, pp.~69--114, 1998.

\bibitem{teel2000smooth}
A.~R. Teel and L.~Praly, ``A smooth {L}yapunov function from a class-estimate involving two positive semidefinite functions,'' {\em ESAIM: Control, Optimisation and Calculus of Variations}, vol.~5, pp.~313--367, 2000.

\bibitem{vannelli1985maximal}
A.~Vannelli and M.~Vidyasagar, ``Maximal {L}yapunov functions and domains of attraction for autonomous nonlinear systems,'' {\em Automatica}, vol.~21, no.~1, pp.~69--80, 1985.

\bibitem{hassan2002nonlinear}
K.~K. Hassan, {\em Nonlinear systems}.
\newblock Pearson, 2001.

\bibitem{giesl2015review}
P.~Giesl and S.~Hafstein, ``Review on computational methods for {L}yapunov functions,'' {\em Discrete \& Continuous Dynamical Systems-B}, vol.~20, no.~8, p.~2291, 2015.

\bibitem{dawson2022safe}
C.~Dawson, S.~Gao, and C.~Fan, ``Safe control with learned certificates: A survey of neural {{L}yapunov}, barrier, and contraction methods for robotics and control,'' {\em IEEE Transactions on Robotics}, pp.~1--19, 2023.

\bibitem{brunton2021modern}
S.~L. Brunton, M.~Budi{\v{s}}i{\'c}, E.~Kaiser, and J.~N. Kutz, ``Modern {K}oopman theory for dynamical systems,'' {\em arXiv preprint arXiv:2102.12086}, 2021.

\bibitem{koopman1931hamiltonian}
B.~O. {K}oopman, ``Hamiltonian systems and transformation in hilbert space,'' {\em Proceedings of the National Academy of Sciences}, vol.~17, no.~5, pp.~315--318, 1931.

\bibitem{mezic2005spectral}
I.~Mezi{\'c}, ``Spectral properties of dynamical systems, model reduction and decompositions,'' {\em Nonlinear Dynamics}, vol.~41, pp.~309--325, 2005.

\bibitem{schmid2009dynamic}
P.~J. Schmid and J.~Sesterhenn, ``Dynamic mode decomposition of experimental data,'' in {\em Proc. of ISPIV}, 2009.

\bibitem{williams2015data}
M.~O. Williams, I.~G. Kevrekidis, and C.~W. Rowley, ``A data--driven approximation of the {K}oopman operator: Extending dynamic mode decomposition,'' {\em Journal of Nonlinear Science}, vol.~25, pp.~1307--1346, 2015.

\bibitem{lusch2018deep}
B.~Lusch, J.~N. Kutz, and S.~L. Brunton, ``Deep learning for universal linear embeddings of nonlinear dynamics,'' {\em Nature Communications}, vol.~9, no.~1, p.~4950, 2018.

\bibitem{azencot2020forecasting}
O.~Azencot, N.~B. Erichson, V.~Lin, and M.~Mahoney, ``Forecasting sequential data using consistent {K}oopman autoencoders,'' in {\em International Conference on Machine Learning}, pp.~475--485, PMLR, 2020.

\bibitem{mauroy2013spectral}
A.~Mauroy and I.~Mezi{\'c}, ``A spectral operator-theoretic framework for global stability,'' in {\em Proc. of CDC}, pp.~5234--5239, IEEE, 2013.

\bibitem{mauroy2016global}
A.~Mauroy and I.~Mezi{\'c}, ``Global stability analysis using the eigenfunctions of the {K}oopman operator,'' {\em IEEE Transactions on Automatic Control}, vol.~61, no.~11, pp.~3356--3369, 2016.

\bibitem{deka2022koopman}
S.~A. Deka, A.~M. Valle, and C.~J. Tomlin, ``{K}oopman-based neural {L}yapunov functions for general attractors,'' in {\em 2022 IEEE 61st Conference on Decision and Control (CDC)}, pp.~5123--5128, IEEE, 2022.

\bibitem{zhou2022neural}
R.~Zhou, T.~Quartz, H.~De~Sterck, and J.~Liu, ``Neural {L}yapunov control of unknown nonlinear systems with stability guarantees,'' {\em Advances in Neural Information Processing Systems}, vol.~35, 2022.

\bibitem{yi2023equivalence}
B.~Yi and I.~R. Manchester, ``On the equivalence of contraction and {K}oopman approaches for nonlinear stability and control,'' {\em IEEE Transactions on Automatic Control}, 2023.

\bibitem{bierwart4569180numerical}
F.-G. Bierwart and A.~Mauroy, ``A numerical {K}oopman-based framework to estimate regions of attraction for general vector fields,'' {\em Available at SSRN 4569180}.

\bibitem{zeng2024sampling}
Z.~Zeng, Z.~Yue, A.~Mauroy, J.~Gon{\c{c}}alves, and Y.~Yuan, ``A sampling theorem for exact identification of continuous-time nonlinear dynamical systems,'' {\em IEEE Transactions on Automatic Control}, 2024.

\bibitem{liu2023towards}
J.~Liu, Y.~Meng, M.~Fitzsimmons, and R.~Zhou, ``Towards learning and verifying maximal neural {L}yapunov functions,'' in {\em Proc. of CDC}, 2023.

\bibitem{grune2021computing}
L.~Gr{\"u}ne, ``Computing {L}yapunov functions using deep neural networks,'' {\em Journal of Computational Dynamics}, vol.~8, no.~2, pp.~131--152, 2021.

\bibitem{raissi2019physics}
M.~Raissi, P.~Perdikaris, and G.~E. Karniadakis, ``Physics-informed neural networks: A deep learning framework for solving forward and inverse problems involving nonlinear partial differential equations,'' {\em Journal of Computational physics}, vol.~378, pp.~686--707, 2019.

\bibitem{kang2021data}
W.~Kang, K.~Sun, and L.~Xu, ``Data-driven computational methods for the domain of attraction and {Z}ubov's equation,'' {\em arXiv preprint arXiv:2112.14415}, 2021.

\bibitem{lan2013linearization}
Y.~Lan and I.~Mezi{\'c}, ``Linearization in the large of nonlinear systems and {K}oopman operator spectrum,'' {\em Physica D: Nonlinear Phenomena}, vol.~242, no.~1, pp.~42--53, 2013.

\bibitem{kvalheim2021existence}
M.~D. Kvalheim and S.~Revzen, ``Existence and uniqueness of global {K}oopman eigenfunctions for stable fixed points and periodic orbits,'' {\em Physica D: Nonlinear Phenomena}, vol.~425, p.~132959, 2021.

\bibitem{zubov1961methods}
V.~I. Zubov, {\em Methods of AM {L}yapunov and Their Application}, vol.~4439.
\newblock US Atomic Energy Commission, 1961.

\bibitem{farsi2023model}
M.~Farsi and J.~Liu, {\em Model-based Reinforcement Learning: From Data to Actions with Python-based Toolbox}.
\newblock Wiley-Blackwell, 2023.

\bibitem{bardi1997optimal}
M.~Bardi and I.~C. Dolcetta, {\em Optimal Control and Viscosity Solutions of Hamilton-Jacobi-Bellman Equations}, vol.~12.
\newblock Springer, 1997.

\bibitem{evans2010partial}
L.~C. Evans, {\em Partial Differential Equations}, vol.~19.
\newblock American Mathematical Society, 2010.

\bibitem{camilli2001generalization}
F.~Camilli, L.~Gr{\"u}ne, and F.~Wirth, ``A generalization of {Z}ubov's method to perturbed systems,'' {\em SIAM Journal on Control and Optimization}, vol.~40, no.~2, pp.~496--515, 2001.

\bibitem{liu2023physics}
J.~Liu, Y.~Meng, M.~Fitzsimmons, and R.~Zhou, ``Physics-informed neural network {L}yapunov functions: {PDE} characterization, learning, and verification,'' {\em arXiv preprint arXiv:2312.09131}, 2023.

\bibitem{zhu2003lower}
Q.~J. Zhu, ``Lower semicontinuous {L}yapunov functions and stability,'' {\em Journal of Nonlinear and Convex Analysis}, vol.~4, p.~325, 2003.

\bibitem{jones2021converse}
M.~Jones and M.~M. Peet, ``Converse {L}yapunov functions and converging inner approximations to maximal regions of attraction of nonlinear systems,'' in {\em Proc. of CDC}, pp.~5312--5319, IEEE, 2021.

\bibitem{acc2022}
Y.~Meng, Y.~Li, and J.~Liu, ``Control of nonlinear systems with reach-avoid-stay specifications: A {L}yapunov-barrier approach with an application to the moore-greizer model,'' in {\em Proc. of ACC}, pp.~2284--2291, 2021.

\bibitem{meng2022smooth}
Y.~Meng, Y.~Li, M.~Fitzsimmons, and J.~Liu, ``Smooth converse {L}yapunov-barrier theorems for asymptotic stability with safety constraints and reach-avoid-stay specifications,'' {\em Automatica}, vol.~144, p.~110478, 2022.

\bibitem{meng2023lyapunov}
Y.~Meng and J.~Liu, ``{L}yapunov-barrier characterization of robust reach--avoid--stay specifications for hybrid systems,'' {\em Nonlinear Analysis: Hybrid Systems}, vol.~49, p.~101340, 2023.

\bibitem{oksendal2013stochastic}
B.~Oksendal, {\em Stochastic Differential Equations: An Introduction with Applications}.
\newblock Springer Science \& Business Media, 2013.

\bibitem{mauroy2019koopman}
A.~Mauroy and J.~Goncalves, ``{K}oopman-based lifting techniques for nonlinear systems identification,'' {\em IEEE Transactions on Automatic Control}, vol.~65, no.~6, pp.~2550--2565, 2019.

\bibitem{Brunton2022survey}
S.~L. Brunton, M.~Budi\v{s}i\'{c}, E.~Kaiser, and J.~N. Kutz, ``Modern {K}oopman theory for dynamical systems,'' {\em SIAM Review}, vol.~64, no.~2, pp.~229--340, 2022.

\bibitem{khromov2023some}
G.~Khromov and S.~P. Singh, ``Some fundamental aspects about lipschitz continuity of neural network functions,'' {\em arXiv preprint arXiv:2302.10886}, 2023.

\bibitem{topcu2010help}
U.~Topcu, A.~Packard, P.~Seiler, and G.~Balas, ``Help on sos [ask the experts],'' {\em IEEE Control Systems Magazine}, vol.~30, no.~4, pp.~18--23, 2010.

\bibitem{gao2013dreal}
S.~Gao, S.~Kong, and E.~M. Clarke, ``{dReal}: An {SMT} solver for nonlinear theories over the reals,'' in {\em Proc. of CADE}, 2013.

\bibitem{blomker2004multiscale}
D.~Bl{\"o}mker and M.~Hairer, ``Multiscale expansion of invariant measures for spdes,'' {\em Communications in mathematical physics}, vol.~251, no.~3, pp.~515--555, 2004.

\bibitem{meng2023hopf}
Y.~Meng, N.~S. Namachchivaya, and N.~Perkowski, ``Hopf bifurcations of moore-greitzer pde model with additive noise,'' {\em Journal of Nonlinear Science}, vol.~33, no.~5, p.~74, 2023.

\bibitem{meng2024koopman}
Y.~Meng, R.~Zhou, M.~Ornik, and J.~Liu, ``{K}oopman-based learning of infinitesimal generators without operator logarithm,'' {\em arXiv preprint arXiv:2403.15688}, 2024.

\end{thebibliography}


\begin{thebibliography}{00}
\bibitem{bib1} G. O. Young, ``Synthetic structure of industrial plastics,'' in {\it Plastics,}2$^{\mathrm{nd}}$ ed., vol. 3, J. Peters, Ed. New York, NY, USA: McGraw-Hill, 1964, pp. 15--64.
\bibitem{bib2} W.-K. Chen, {\it Linear Networks and Systems.}Belmont, CA, USA: Wadsworth, 1993, pp. 123--135.
\end{thebibliography}

\begin{thebibliography}{00}\leftskip1pc
\bibitem{bib3} J. U. Duncombe, ``Infrared navigation---Part I: An assessment of feasibility,'' {\it IEEE Trans. Electron Devices}, vol. ED-11, no. 1, pp. 34--39, Jan. 1959, doi:.  {10.1109/TED.2016.2628402}.
\bibitem{bib4} E. P. Wigner, ``Theory of traveling-wave optical laser,''
{\it Phys. Rev}., vol. 134, pp. A635--A646, Dec. 1965, doi:  {10.1109.} {{\it XXX}} {.123456}.
\bibitem{bib5} E. H. Miller, ``A note on reflector arrays,'' {\it IEEE Trans. Antennas Propagat}., to be published.
\end{thebibliography}

\begin{thebibliography}{00}\leftskip1pc
\bibitem{bib6} E. E. Reber, R. L. Michell, and C. J. Carter, ``Oxygen absorption in the earth's atmosphere,'' Aerospace Corp., Los Angeles, CA, USA, Tech. Rep. TR-0200 (4230-46)-3, Nov. 1988.
\bibitem{bib7} J. H. Davis and J. R. Cogdell, ``Calibration program for the 16-foot antenna,'' Elect. Eng. Res. Lab., Univ. Texas, Austin, TX, USA, Tech. Memo. NGL-006-69-3, Nov. 15, 1987.
\end{thebibliography}

\begin{thebibliography}{00}\leftskip1pc
\bibitem{bib8} {\it Transmission Systems for Communications}, 3rd ed., Western Electric Co., Winston-Salem, NC, USA, 1985, pp. 44--60.
\bibitem{bib9} {\it Motorola Semiconductor Data Manual}, Motorola Semiconductor Products Inc., Phoenix, AZ, USA, 1989.
\end{thebibliography}

\begin{thebibliography}{00}\leftskip1pc
\bibitem{bib10} G. O. Young, ``Synthetic structure of industrial plastics,'' in Plastics, vol. 3, Polymers of Hexadromicon, J. Peters, Ed., 2nd ed. New York, NY, USA: McGraw-Hill, 1964, pp. 15-64. [Online]. Available:  {http://www.bookref.com}.
\bibitem{bib11} {\it The Founders' Constitution}, Philip B. Kurland and Ralph Lerner, eds., Chicago, IL, USA: Univ. Chicago Press, 1987. [Online]. Available:  {http://press-pubs.uchicago.edu/founders/}
\bibitem{bib12} The Terahertz Wave eBook. ZOmega Terahertz Corp., 2014. [Online]. Available:  {http://dl.z-thz.com/eBook/zomega\_ebook\_pdf\_1206\_sr.pdf}. Accessed on: May 19, 2014.
\bibitem{bib13} Philip B. Kurland and Ralph Lerner, eds., {\it The Founders' Constitution.}Chicago, IL, USA: Univ. of Chicago Press, 1987, Accessed on: Feb. 28, 2010, [Online] Available:  {http://press-pubs.uchicago.edu/founders/}
\end{thebibliography}

\begin{thebibliography}{00}\leftskip1pc
\bibitem{bib14} J. S. Turner, ``New directions in communications,'' {\it IEEE J. Sel. Areas Commun}., vol. 13, no. 1, pp. 11-23, Jan. 1995. DOI.  {10.1109.} {{\it XXX}} {.123456}.
\bibitem{bib15} W. P. Risk, G. S. Kino, and H. J. Shaw, ``Fiber-optic frequency shifter using a surface acoustic wave incident at an oblique angle,'' {\it Opt. Lett.}, vol. 11, no. 2, pp. 115--117, Feb. 1986, doi: {10.1109.} {{\it XXX}} {.123456}.
\bibitem{bib16} P. Kopyt {\it \textit{et al.}, ``}Electric properties of graphene-based conductive layers from DC up to terahertz range,'' {\it IEEE THz Sci. Technol.,}to be published, doi:  {10.1109/TTHZ.2016.2544142}.
\end{thebibliography}

\begin{thebibliography}{00}\leftskip1pc
\bibitem{bib17} PROCESS Corporation, Boston, MA, USA. Intranets: Internet technologies deployed behind the firewall for corporate productivity. Presented at INET96 Annual Meeting. [Online]. Available:  {http://home.process.com/Intranets/wp2.htp}
\end{thebibliography}

\begin{thebibliography}{00}\leftskip1pc
\bibitem{bib18} R. J. Hijmans and J. van Etten, ``Raster: Geographic analysis and modeling with raster data,'' R Package Version 2.0-12, Jan. 12, 2012. [Online]. Available:  {http://CRAN.R-project.org/package$=$raster} {}
\bibitem{bib19} Teralyzer. Lytera UG, Kirchhain, Germany [Online]. Available: http://www.lytera.de/Terahertz\_THz\_Spectroscopy.php?id$=$home, Accessed on: Jun. 5, 2014
\end{thebibliography}

\begin{thebibliography}{00}\leftskip1pc
\bibitem{bib20} U.S. House. 102nd Congress, 1st Session. (1991, Jan. 11). {\it H. Con. Res. 1, Sense of the Congress on Approval of Military Action}. [Online]. Available: LEXIS Library: GENFED File: BILLS
\end{thebibliography}

\begin{thebibliography}{00}\leftskip1pc
\bibitem{bib21} Musical toothbrush with mirror, by L.M.R. Brooks. (1992, May 19). Patent D 326 189 [Online]. Available: NEXIS Library: LEXPAT File: DES
\end{thebibliography}

\begin{thebibliography}{00}\leftskip1pc
\bibitem{bib22} D. B. Payne and J. R. Stern, ``Wavelength-switched passively coupled single-mode optical network,'' in {\it Proc. IOOC-ECOC,}Boston, MA, USA,  1985,
pp. 585--590, doi:  {10.1109.} {{\it XXX}} {.123456}.
\end{thebibliography}

\begin{thebibliography}{00}\leftskip1pc
\bibitem{bib23} D. Ebehard and E. Voges, ``Digital single sideband detection for interferometric sensors,'' presented at the {\it 2nd Int. Conf. Optical Fiber Sensors,} Stuttgart, Germany, Jan. 2-5, 1984.
\end{thebibliography}

\begin{thebibliography}{00}\leftskip1pc
\bibitem{bib24} G. Brandli and M. Dick, ``Alternating current fed power supply,'' U.S. Patent 4 084 217, Nov. 4, 1978.
\end{thebibliography}

\begin{thebibliography}{00}\leftskip1pc
\bibitem{bib25} J. O. Williams, ``Narrow-band analyzer,'' Ph.D. dissertation, Dept. Elect. Eng., Harvard Univ., Cambridge, MA, USA, 1993.
\bibitem{bib26} N. Kawasaki, ``Parametric study of thermal and chemical nonequilibrium nozzle flow,'' M.S. thesis, Dept. Electron. Eng., Osaka Univ., Osaka, Japan, 1993.
\end{thebibliography}

\begin{thebibliography}{00}\leftskip1pc
\bibitem{bib27} A. Harrison, private communication, May 1995.
\bibitem{bib28} B. Smith, ``An approach to graphs of linear forms,'' unpublished.
\bibitem{bib29} A. Brahms, ``Representation error for real numbers in binary computer arithmetic,'' IEEE Computer Group Repository, Paper R-67-85.
\end{thebibliography}

\begin{thebibliography}{00}\leftskip1pc
\bibitem{bib30} IEEE Criteria for Class IE Electric Systems, IEEE Standard 308, 1969.
\bibitem{bib31} Letter Symbols for Quantities, ANSI Standard Y10.5-1968.
\end{thebibliography}

\begin{thebibliography}{00}\leftskip1pc
\bibitem{bib32} R. Fardel, M. Nagel, F. Nuesch, T. Lippert, and A. Wokaun, ``Fabrication of organic light emitting diode pixels by laser-assisted forward transfer,'' {\it Appl. Phys. Lett.}, vol. 91, no. 6, Aug. 2007, Art. no. 061103, doi:  {10.1109.} {{\it XXX}} {.123456}.
\bibitem{bib33} J. Zhang and N. Tansu, ``Optical gain and laser characteristics of InGaN quantum wells on ternary InGaN substrates,'' {\it IEEE Photon. J.}, vol. 5, no. 2, Apr. 2013, Art. no. 2600111, doi:  {10.1109.} {{\it XXX}} {.123456}.
\end{thebibliography}

\begin{thebibliography}{00}\leftskip1pc
\bibitem{bib34} S. Azodolmolky~{{et al.}}, Experimental demonstration of an impairment aware network planning and operation tool for transparent/translucent optical networks,''~{\it J. Lightw. Technol.}, vol. 29, no. 4, pp. 439--448, Sep. 2011,doi:  {10.1109.} {{\it XXX}} {.123456}.
\end{thebibliography}

\begin{thebibliography}{00}\leftskip1pc
\bibitem{bib35} U.S. Department of Health and Human Services, Aug. 2013, ``Treatment Episode Dataset: Discharges (TEDS-D): Concatenated, 2006 to 2009,'' U.S. Department of Health and Human Services, Substance Abuse and Mental Health Services Administration, Office of Applied Studies, doi: 10.3886/ICPSR30122.v2.
\end{thebibliography}

\begin{thebibliography}{00}\leftskip1pc
\bibitem{bib36} T. D'Martin and S. Soares, 2019, ``Code for Assessment of Markov Decision Processes in Long-Term Hydrothermal Scheduling of Single-Reservoir Systems (Version 1.0),'' Code Ocean, doi: \_1.24433/CO.7212286.v1
\end{thebibliography}

\end{document}